%% file: article.tex
\documentclass[11pt]{amsart}
\usepackage{a4wide}
\usepackage{hyperref}
\usepackage{bookmark}
\parskip=\baselineskip

\let\oldlist=\list
\newlength\oldparskip
\def\list#1#2{\oldparskip=\parskip\parskip=0.5\baselineskip
\oldlist{#1}{#2}\parskip=0.5\baselineskip}
\let\oldendlist=\enditemize
\def\enditemize{\oldendlist\vspace*{-0.5\oldparskip}}

\input{sdpDefinitions}

\begin{document}
\title[Permanence properties of property A and coarse embeddability]{Permanence properties of property A and coarse embeddability for locally compact groups}
\newbox\sdptempbox
\def\hang#1{\setbox\sdptempbox=\hbox{#1}\hspace*{-\wd\sdptempbox}\box\sdptempbox}
\author{Steven Deprez$^{1}$ and Kang Li$^{2}$}
\thanks{\hang{$^{1}$ }corresponding author, University of Copenhagen, sdeprez@math.ku.dk}
\thanks{Supported by ERC Advanced Grant no. OAFPG 247321}
\thanks{\hang{$^{2}$ }University of Copenhagen, kang.li@math.ku.dk}
\thanks{\hang{$^{1,2}$ }Supported by the Danish National Research Foundation through the Centre for Symmetry and Deformation}
\thanks{(DNRF92).} 

\def\lcsc{l.c.s.c\wordspace}
\def\wordspace{\futurelet\sdptoken\dowordspace}
\def\dowordspace{\ifcat\sdptoken a.\@ \else\if\sdptoken.\else\if\sdptoken,\else\if\sdptoken)\else\undefined\fi\fi\fi\fi}

\begin{abstract}
  If $\Gamma\subset G$ is a lattice in a locally compact second countable group $G$, then we show that $G$ has property A (respectively is coarsely embeddable into Hilbert space)
  if and only if $\Gamma$ has property A (respectively is coarsely embeddable into Hilbert space).
  Moreover, we show three interesting generalizations of this result.
  If $H\subset G$ is a closed subgroup of $G$ that is co-amenable in $G$,
  and if $H$ has property A (respectively, is coarsely embeddable into Hilbert space),
  then we show that $G$ has property A (respectively, is coarsely embeddable into Hilbert space).
  We also show that an extension of property A groups still has property A.
  On the coarse embeddability side, we show that if $e\rightarrow H\rightarrow G\rightarrow Q\rightarrow e$ is a short exact sequence, and if
  either $H$ is coarsely embeddable into Hilbert space and $Q$ has property A, or $H$ is compact and $Q$ is coarsely embeddable into Hilbert space,
  then $G$ is coarsely embeddable into Hilbert space.
  We extend the theory of measure equivalence to locally compact non-unimodular groups.
  In a natural way, we can also define measure equivalence subgroups.
  We show that property A and uniform embeddability into Hilbert space pass to measure equivalence subgroups.
  Using the same techniques, we show that also the Haagerup property,
  weak amenability and the weak Haagerup property pass to measure equivalence subgroups.
\end{abstract}

\maketitle
\section*{Introduction and statement of the main results}
In \cite{G93}, Gromov introduced the notion of uniform embeddability of metric spaces. Nowadays, this is often called coarse embeddability, and
we stick to the more modern terminology in this paper. Gromov suggested that a discrete finitely generated group $\Gamma$
that coarsely embeds into Hilbert space, would satisfy the Novikov conjecture.
It was later shown that this is indeed the case: in \cite{Yu00}, Yu
showed that it is true for all discrete groups that are uniformly embeddable into Hilbert space, and whose classifying space $B\Gamma$ is a finite CW-complex.
In the same paper, he introduced a
condition on $\Gamma$, which he called property A, that ensures coarse embeddability of $\Gamma$ into a Hilbert space.
Higson and Roe
showed that $\Gamma$ has property A if and only if $\Gamma$ has a topologically amenable action on a compact Hausdorff space \cite{HR00}. Ozawa
showed that this is equivalent to exactness \cite{Oza00}.
In \cite{H00}, Higson showed that all discrete groups with property A satisfy the Novikov conjecture, even when the classifying space
is not a finite CW-compex. Skandalis, Tu and Yu \cite{STY02} could then show that
all discrete groups that coarsely embed into Hilbert space, do indeed satisfy the Novikov conjecture. In fact,
the results by Higson and Skandalis, Tu and Yu are slightly stronger: they showed that for all discrete groups with property A (respectively that are coarsely embeddable into Hilbert space),
the Baum-Connes assembly map with coefficients is split-injective. Baum, Connes and Higson showed in \cite{BCH94} that this implies the Novikov conjecture.

Similar results also hold for locally compact second countable (from now on \lcsc) groups. In \cite{R05}, Roe extended the definition of
property A to proper metric spaces with bounded geometry. Every \lcsc group has a proper compatible left-invariant metric, and has bounded geometry
with respect to this metric (see \cite{S74, HP06}).
So Roe's definition applies in particular to \lcsc groups.
In a previous paper \cite{DL:A}, we showed that a \lcsc group $G$ has property A in this sense if and only if
$G$ admits a topologically amenable action on a compact Hausdorff space. Chabert, Echterhoff and Oyono-Oyono showed that the Baum-Connes assembly map
with coefficients is split-injective for every \lcsc group that admits a topologically amenable action on a compact Hausdorff space, see \cite{CEO04}.
In \cite{DL:A}, we
extended this result and showed that coarse embeddability into Hilbert space implies split-injectivity of the Baum-Connes assembly map with coefficients.

In the present paper, we study permanence properties of property A and coarse embeddability into Hilbert space for \lcsc groups.
For brevity, we will drop the ``into Hilbert space'',
so the phrase ``$G$ is coarsely embeddable'' will mean that $G$ is coarsely embeddable into Hilbert space. For the convenience of the reader,
we review the relevant definitions in section \ref{sect:defs}. It is clear that property A and coarse embeddability pass to closed
subgroups of \lcsc groups. We will be interested in the other direction: if $G$ is a \lcsc group and a closed subgroup
$H\subset G$ has property A (resp.\@ is coarsely embeddable), under which conditions on the inclusion $H\subset G$ can we conclude that $G$
has property A (resp.\@ is coarsely embeddable)? The first of our results of this type is that this holds when $H$ is a lattice in $G$.
This result can be extended in a number of ways, and we obtain the following result.
\begin{theorem}
  \label{thm:main:ex}
  Let $G,H$ be \lcsc groups. In each of the following situations, if $H$ has property A (resp.\@ is coarsely embeddable),
  then $G$ has property A (resp.\@ is coarsely embeddable).
  \begin{enumerate}
  \item $H\subset G$ is a lattice
  \item $H\subset G$ is a closed subgroup with finite covolume
  \item $H\subset G$ is a closed co-amenable subgroup, in the sense of Eymard \cite{Eymard:Means}
  \item $H$ is a closed normal subgroup of $G$ and the quotient group $G/H$ has property A
  \item $H=G/Q$ where $Q\subset G$ is a compact normal subgroup
  \item $G$ is a measure equivalence subgroup of $H$. We give a careful definition of this notion in definition \ref{def:ME}, inspired by \cite{G93}
  \end{enumerate}
\end{theorem}

All of the above statements are special cases of a more general result. The crucial ingredients of this result are \emph{proper cocycles}
(inspired by Jolissaint \cite{Jolissaint:propercocycle}) and property A for pairs
(inspired by amenable pairs \cite{Eymard:Means, Greenleaf:AmenableActions, Zimmer:AmenablePairs, Jolissaint:AmenablePairs}).
We explain both notions below.

In \cite{Jolissaint:propercocycle}, Jolissaint introduced proper cocycles in order to prove permanence
properties of the Haagerup property for \lcsc groups.
Recently, he used proper cocycles to derive similar permanence properties of weak amenability and of the weak
Haagerup property for \lcsc groups \cite{Jolissaint:permanence}.
We introduce a slightly weaker notion of proper cocycle. 
\begin{definition}[inspired by \cite{Jolissaint:propercocycle}]
  Let $G,H$ be \lcsc groups and let $G\actson (X,\mu)$ be a non-singular Borel action on a standard probability
  space. A Borel cocycle $\omega:G\times X\rightarrow H$ is said
  to be
  \begin{itemize}
  \item \emph{proper with respect to a family $\cA$ of Borel sets in $X$} if
    \begin{enumerate}
    \item for every compact subset $K\subset G$ and every $A,B\in\cA$, we find a precompact set
      $L(K,A,B)\subset H$ such that, for every $g\in K$ we get that
      \[\omega(g,x)\in L(K,A,B) \text{ for almost all }x\in A\cap g^{-1}B\]
    \item for every compact subset $L\subset H$ and every $A,B\in\cA$, we get that the set $K(L,A,B)$
      of all $g\in G$ such that
      \[\mu\{x\in X\mid x\in A, gx\in B, \omega(g,x)\in L\}>0,\]
      is precompact in $G$.
    \end{enumerate}
  \item \emph{proper} if $\omega$ is proper with respect to some family $\cA$ of Borel sets in $X$ such that for every $\varepsilon>0$ there is a set $A\in\cA$ with $\mu(X\setminus A)<\varepsilon$.
  \end{itemize}
\end{definition}
Our notion of proper cocycle has a few advantages over Jolissaint's notion. First of all, it is more natural
because it is invariant under cohomology, while Jolissaint's notion is not. In section \ref{sect:propercocycle},
we give an example of a cocycle $\omega$ that is cohomologous to a cocycle that is proper in Jolissaint's sense,
but $\omega$ itself is not proper in Jolissaint's sense (see example \ref{ex:non-J}).
More importantly, we have more examples. In theorem \ref{thm:MC},
we show that the cocycles coming from measure equivalence subgroups are proper in our sense.
Every cocycle that is proper in Jolissaint's sense is also proper in our sense, see proposition \ref{prop:J-our}.
So example \ref{ex:non-J} gives an example of a cocycle that is proper in our sense but not in Jolissaint's sense.
Jolissaint's main result about proper cocycles is \cite[theorem 1.4]{Jolissaint:permanence}.
With theorem \ref{thm:amen:intro} below, we show that this result also holds for our weaker notion of proper cocycle.

\def\niets{
 In full
generality, these cocycles are not proper in Jolissaint's sense, as is indicated by the following very simple
example. Consider the action of $\IR$ on $\Circle^1$ by rotation. Assume that the stabilizer of any point in
$\Circle^1$ is $\IZ$. Let $\varphi:S^1\rightarrow \IR$ be any unbounded Borel function. Then the cocycle
$\omega:\IR\times \Circle^1\rightarrow \IR$ that is defined by $\omega(r,x)=\varphi(rx)r\varphi(x)^{-1}$
is proper in our sense, but not in Jolissaint's sense. This cocycle comes from a measure equivalence between
$\IR$ and itself.
 However, every cocycle that is proper in Jolissaint's sense is also proper in our sense,
and the Jolissaint's results that are based on the existence of proper cocycles remain true, see theorem
\ref{thm:amen:intro}.}

From now on, when we say that $G\actson (X,\mu)$ is a non-singular action, we mean that $G\actson (X,\mu)$ is a non-singular Borel action of a \lcsc group on a standard probability space.
Moreover, we assume all cocycles to be Borel cocycles.

\begin{examples}[\cite{Jolissaint:propercocycle}]
  Jolissaint gives the following elementary examples of proper cocycles.
  \begin{itemize}
  \item Let $H\subset G$ be a closed subgroup of a \lcsc group. Set $X=G/H$ and consider any quasi-invariant probability measure $\mu$ on $X$. Let $s:X\rightarrow G$
    be a regular Borel section (see \cite[lemma 1.1]{Mackey:RegularSection}), i.e.\@ $s(x)H=x$ for all $x\in X$ and $s(K)$ is precompact in $G$, for all compact subsets $K\subset X$. Then the cocycle $\omega:G\times X\rightarrow H$
    that is defined by $\omega(g,x)=s(gx)^{-1}gs(x)$ is a proper cocycle with respect to the family of all compact subsets of $X$.
  \item Let $Q\subset G$ be a closed normal subgroup and set $H=G/Q$. Consider $X$ to be the one-point space and consider the quotient morphism $\pi:G\rightarrow H$ to be a cocycle
    $\pi:G\times X\rightarrow H$. Such a cocycle is proper if and only if $Q$ is compact.
  \end{itemize}
\end{examples}

We add a non-obvious example to the list, coming from measure equivalence of locally compact groups.
Measure equivalence was introduced for discrete groups by Gromov in \cite[section 0.5.E]{G93},
as a measure-theoretic counterpart to coarse equivalence. Even though he only mentions discrete
groups, his definition works well for unimodular \lcsc groups, see for example
\cite{FurmanBaderSauer:MErigidity}. In the non-unimodular case, we need to take a little more care.
We first introduce the much weaker notion of a measure correspondence:
we say that a non-singular action $G\times H\actson (\Omega,\eta)$ is a \emph{measure correspondence} between
$G$ and $H$ if there are standard probability spaces $(X,\mu)$ and $(Y,\nu)$ and measure class preserving
Borel isomorphisms $\varphi:X\times H\rightarrow \Omega$ and $\psi:Y\times G\rightarrow \Omega$
such that $\varphi$ commutes with the $H$ action and $\psi$ commutes with the $G$ action.
We consider the Haar measure on the groups $G$ and $H$.

Between every two \lcsc groups $G,H$, there is a measure correspondence, namely $\Omega=G\times H$ with the left
translation action of $G\times H$. There is also a ``composition'' operation on measure correspondences:
if $\Omega_1$ is a correspondence between $G,H$ and $\Omega_2$ is a measure correspondence between $H,K$, then
$\Omega_1\otimes_H\Omega_2$ is a measure correspondence between $G$ and $K$.

Given a measure correspondence $\Omega$ between $G$ and $H$, we can transfer the $G$-action from
$\Omega$ to $X\times H$, using $\varphi$. This action is of the form
\[g\varphi(x,h)=\varphi(gx,h\omega(g,x)^{-1})\text{ for all }g\in G\text{ and almost all }(x,h)\in X\times H,\]
for some non-singular action $G\actson X$ and some Borel cocycle $\omega:G\times X\rightarrow H$. In a similar way, we find an action of $H$ on $Y$
and a cocycle $\beta:H\times Y\rightarrow G$.
We say that $G$ is a measure equivalence subgroup of $H$ if there is a $G$-invariant probability measure on $X$ in the measure class of $\mu$.
We say that $G$ is measure equivalent to $H$ if
there are invariant probability measures on $X$,$Y$, in the measure class of $\mu$,$\nu$.
If $G,H$ are unimodular, then this is equivalent to the usual notion of measure equivalence, see theorem \ref{thm:MEdef}.
\begin{example}[see theorem \ref{thm:MC}]
  \label{ex:intro:MC}
  Let $G,H$ be \lcsc groups and let $(\Omega,\eta)$ be a measure correspondence between $G$ and $H$. Let $\omega:G\times X\rightarrow H$ be a cocycle as defined above. Then $\omega$
  is a proper cocycle.
\end{example}

Property A for pairs is based on amenability for pairs. Let $G\actson (X,\mu)$ be a non-singular action and denote its Radon-Nikodym cocycle by $\chi:G\times X\rightarrow \IRpos$.
Consider the Koopman representation $\pi_X:G\rightarrow\cU(\Lp^2(X,\mu))$, i.e.\@ $(\pi_X(g)\xi)(x)=\xi(g^{-1}x)\sqrt{\chi(g^{-1},x)}$ for all $g\in G$, $\xi\in\Lp^2(X,\mu)$ and almost all $x\in X$.
Remember that the pair $(G,X)$ is said to be amenable if the Koopman representation $\pi_X$ has almost invariant vectors (see \cite{Eymard:Means, Greenleaf:AmenableActions, Zimmer:AmenablePairs, Jolissaint:AmenablePairs}
for more background and equivalent definitions). In this spirit, we make the following definition.
\begin{definition}
  Let $G\actson (X,\mu)$ be a non-singular action and let $\cA$ be a family of Borel subsets of $X$. We say that the pair $(G,X)$ has property A with respect to the family $\cA$ if
  for every compact set $K\subset G$ and every $\varepsilon>0$, there exists a continuous family $(\xi_g)_{g\in G}$ of unit vectors in $\Lp^2(X,\mu)$ such that
  \begin{itemize}
  \item $\norm{\xi_g-\xi_h}_2<\varepsilon$ whenever $g^{-1}h\in K$
  \item there is a set $A\in\cA$ such that each $\xi_g$ is supported in $gA$.
  \end{itemize}
\end{definition}

It is clear that every amenable pair has property A: take a $(K,\varepsilon)$-invariant vector $\xi\in\Lp^2(X,\mu)$.
We can assume that $\xi$ is supported in $A$ for some set $A\in\cA$. Set now $\xi_g=\pi_X(g)\xi$.
In particular, if the measure class of $\mu$ contains a $G$-invariant probability measure, then we know that the pair $(G,X)$
is amenable and hence has property A.
When $X$ is of the form $G/H$ where $H$ is a closed normal subgroup, then the pair $(G,G/H)$ has property A with respect to the family of compact sets in $G/H$
if and only if the group $G/H$ has property A as a locally compact group. Slightly more generally, when $H\subset G$ is a closed subgroup and $X=G/H$,
then the pair $(G,X)$ has property A with respect to the family of all compact sets in $G/H$ if and only if $X=G/H$ has property A in the sense of Roe \cite{R05},
as a metric space.

In \cite{Jolissaint:permanence}, Jolissaint showed the following theorem. He proved the case with the Haagerup property
before in \cite{Jolissaint:propercocycle}. The statement about the weak Haagerup property was shown before by Knudby \cite{Knudby:WH}, in the
case where the action preserves an infinite measure and satisfies the F\o{}lner condition.
\begin{theorem}[{\cite[theorem 1.4]{Jolissaint:permanence}}, see theorem \ref{thm:amen:proof} for our proper cocycle]
  \label{thm:amen:intro}
  Let $G,H$ be \lcsc groups, let $G\actson (X,\mu)$ be a non-singular action such that the pair $(G,X)$
  is amenable. Let $\omega:G\times X\rightarrow H$
  be a proper cocycle. If $H$ has the Haagerup property, is weakly amenable, respectively has the weak
  Haagerup property, then so does $G$. Moreover, the weak amenability and weak Haagerup constant of $G$ is less
  than that of $H$:
  \[\Lambda_{WA}(G)\leq \Lambda_{WA}(H)\text{ and }\Lambda_{WH}(G)\leq\Lambda_{WH}(H).\]
\end{theorem}
Together with example \ref{ex:intro:MC}, this shows that the Haagerup property, weak amenability and the weak
Haagerup property for locally compact groups pass to measure equivalence subgroups.

The main result of this paper is the following theorem.
\begin{theorem}[see theorem \ref{thm:main:proof}]
  \label{thm:main:intro}
  Let $G,H$ be \lcsc groups and let $G\actson (X,\mu)$ be a non-singular action. Suppose that $\omega:G\times X\rightarrow H$ is a proper cocycle with respect to some family $\cA$
  and that the pair $(G,X)$ has property A with respect to the same family $\cA$.
  If $H$ has property A (respectively is coarsely embeddable), then $G$ has property A (respectively is coarsely embeddable).
\end{theorem}
It is clear from the above that theorem \ref{thm:main:ex} is a direct consequence of theorem \ref{thm:main:intro}.

\section{Preliminaries on approximation properties}
\label{sect:defs}
Let $G$ be a \lcsc group. A continuous function $\varphi:G\rightarrow \IC$ is called a Herz-Schur multiplier if there exist bounded continuous functions
$\xi,\eta:G\rightarrow \cH$ from $G$ into Hilbert space $\cH$, such that
\[\varphi(g^{-1}h)=\inprod{\xi_g}{\eta_h}\text{ for all }g,h\in G.\]
For a bounded function $\xi:G\rightarrow \cH$, we set $\norm{\xi}_{\infty}=\sup_{g\in G}\norm{\xi_g}$.
The Herz-Schur norm $\norm{\varphi}_{HS}$ of a Herz-Schur multiplier $\varphi:G\rightarrow \IC$ is the infimum of $\norm{\xi}_{\infty}\norm{\eta}_\infty$
where $\xi,\eta$ runs over all pairs of continuous bounded functions $\xi,\eta:G\rightarrow \cH$ from $G$ into Hilbert space $\cH$,
satisfying $\varphi(g^{-1}h)=\inprod{\xi_g}{\eta_h}$ for all $g,h\in G$.

Suppose that $\xi,\eta:G\rightarrow \cH$ are bounded Borel maps from $G$ into a separable Hilbert space and that
there is a function $\varphi:G\rightarrow \IC$ such that $\varphi(g^{-1}h)=\inprod{\xi_g}{\eta_h}$. Then it follows automatically that $\varphi$
is continuous, so it is a Herz-Schur multiplier. Moreover, its Herz-Schur norm is bounded above by $\norm{\varphi}_{HS}\leq \norm{\xi}_\infty\norm{\eta}_\infty$.
So in the following statements, it does not matter if we take continuous maps or Borel maps. This was proven in the unpublished manuscript \cite{Haagerup:CBAP}.
Knudby included this proof also in his paper \cite[lemma C.1]{Knudby:WH}.

One of the possible definitions of amenability is the following.
\begin{definition}
  A \lcsc group $G$ is called amenable if there exists a sequence $(\varphi_n)_n$ of continuous compactly supported
  positive type functions $\varphi_n:G\rightarrow\IC$ such that $\varphi_n(g)$ converges to $1$ uniformly on compact subsets of $G$.
\end{definition}

If in the above definition, you replace the condition that $\varphi_n$ is compactly supported by the condition that $\varphi_n$ tends to $0$ at infinity,
you get the Haagerup property. If instead you fix a constant $C\geq 1$ and replace the condition that $\varphi_n$ is of positive type by the condition that $\varphi_n$ is
a Herz-Schur multiplier with Herz-Schur norm bounded by $\norm{\varphi_n}_{HS}\leq C$, then you get the definition of weak amenability with weak amenability constant (Cowling-Haagerup constant) less
than $C$. Recently, Knudby \cite{Knudby:WH} studied the weak Haagerup property, where you apply both replacements above.
\begin{definition}
  Let $G$ be a \lcsc group.
  \begin{itemize}
  \item We say that $G$ has the Haagerup property if there exists a sequence $(\varphi_n)_n$ of positive type functions $\varphi_n\in C_0(G)$ such that
    $\varphi_n(g)$ converges to $1$ uniformly on compact subsets of $G$.
  \item We say that $G$ is weakly amenable if there exists a constant $C\geq 1$ and a sequence $(\varphi_n)_n$ of continuous compactly supported Herz-Schur multipliers such that $\varphi_n(g)$
    converges to $1$ uniformly on compact subsets of $G$, and such that $\norm{\varphi_n}_{HS}\leq C$ for all $n\in\IN$. The infimum of all such $C$ is called the weak amenability constant $\Lambda_{WA}(G)$
    of $G$. This constant is also called the Cowling-Haagerup constant of $G$ and is often denoted by $\Lambda(G)$.
  \item We say that $G$ is weakly Haagerup if there exists a constant $C\geq 1$ and a sequence $(\varphi_n)_n$ of continuous Herz-Schur multipliers $\varphi_n\in C_0(G)$ such that $\varphi_n(g)$
    converges to $1$ uniformly on compact subsets of $G$, and such that $\norm{\varphi_n}_{HS}\leq C$ for all $n\in\IN$. The infimum of all such $C$ is called the weak Haagerup constant $\Lambda_{WH}(G)$
    of $G$.
  \end{itemize}
\end{definition}

In \cite{DL:A}, we introduced several equivalent definitions of property A and coarse embeddability of \lcsc groups. We review the definitions that are relevant for the current paper.
\begin{definition}[see {\cite[theorem 2.3]{DL:A}}]
  \label{def:A}
  Let $G$ be a \lcsc group with left Haar measure $\mu$. We say that $G$ has property A if one of the following equivalent conditions holds
  \begin{enumerate}
  \item \label{def:A:L2}For every compact subset $K\subset G$ and any $\varepsilon>0$, there is a continuous family $(\xi_g)_{g\in G}$ of unit vectors
    in $\Lp^2(G,\mu)$ such that
    \begin{itemize}
    \item $\norm{\xi_g-\xi_h}<\varepsilon$ whenever $g^{-1}h\in K$
    \item there is a compact set $L\subset G$ such that each $\xi_g$ is supported in $gL$.
    \end{itemize}
  \item \label{def:A:pos}For every compact subset $K\subset G$ and any $\varepsilon>0$ there is a continuous positive definite kernel $k:G\times G\rightarrow \IC$ such that
    \begin{itemize}
    \item $\abs{k(g,h)-1}<\varepsilon$ whenever $g^{-1}h\in K$
    \item the set $\{g^{-1}h\in G\mid k(g,h)\not=0\}$ is precompact.
    \end{itemize}
  \end{enumerate}
\end{definition}

\begin{definition}[see {\cite[theorem 3.4]{DL:A}}]
  \label{def:CE}
  Let $G$ be a \lcsc group. We say that $G$ is coarsely embeddable (into a Hilbert space) if one of the following equivalent conditions holds
  \begin{enumerate}
  \item \label{def:CE:emb}there exists a Hilbert space $H$ and a continuous map $u:G\rightarrow H$ such that
    \begin{itemize}
    \item For every compact set $K\subset G$, there is a real number $R>0$ such that\\ $\norm{u(g)-u(h)}<R$ for all $g,h\in G$ with $g^{-1}h\in K$.
    \item For every real number $R>0$, the set $\{g^{-1}h\in G\mid \norm{u(g)-u(h)}<R\}$ is precompact.
    \end{itemize}
  \item \label{def:CE:CondNeg}there exists a continuous conditionally negative definite kernel $k:G\times G\rightarrow\IC$ such that
    \begin{itemize}
    \item $k$ is bounded on tubes, i.e.\@ for every compact subset $K\subset G$ there is a real number $R>0$ such that $k(g,h)<R$ for all $g,h\in G$ with $g^{-1}h\in K$
    \item $k$ is proper, i.e.\@ for every real number $R>0$, the set $\{g^{-1}h\in G\mid k(g,h)<R\}$ is precompact.
    \end{itemize}
  \end{enumerate}
\end{definition}

By Schoenberg's theorem, we can translate definition \ref{def:CE}.\ref{def:CE:CondNeg} to a statement that ressembles definition \ref{def:A}.\ref{def:A:pos}.
The proof is a standard application of Schoenberg's theorem, see for example \cite[theorem 2.1.1]{HaagerupProperty}. We give a complete proof, for the convenience of the readers that are
not familiar with that argument.
\begin{theorem}
  \label{thm:CEkern}
  A \lcsc group $G$ is coarsely embeddable iff.\@
  for every compact subset $K\subset G$ and any $\varepsilon>0$ there is a continuous positive definite kernel $k:G\times G\rightarrow \IC$ such that
  \begin{itemize}
  \item $\abs{k(g,h)-1}<\varepsilon$ whenever $g^{-1}h\in K$
  \item for every $\delta>0$, the set $\{g^{-1}h\in G\mid k(g,h)>\delta\}$ is precompact.
  \end{itemize}
\end{theorem}
\begin{proof}
  Let $k:G\times G\rightarrow \IR$ be a conditionally negative definite kernel as in definition \ref{def:CE}.\ref{def:CE:CondNeg}.
  Observe that $k$ takes positive values.
  Let $K\subset G$ be a compact set and let $\varepsilon>0$.
  By Since $k$ is bounded on tubes, there is a real number $R>0$ such that $k(g,h)<R$ for all $g,h\in G$ with $g^{-1}h\in K$.
  Let $t>0$ be such that $1-\exp(-tR)<\varepsilon$. Schoenberg's theorem asserts that the kernel $k_0(g,h)=\exp(-tk(g,h))$ is positive definite. It is easy to see that
  \[\abs{k_0(g,h)-1}=1-\exp(-tk(g,h))\leq 1-\exp(-tR)<\varepsilon\]
  whenever $g^{-1}h\in K$. Moreover, let $\delta>0$. Since $k$ is a proper kernel, we know that the set 
  \[\{g^{-1}h\mid \abs{k_0(g,h)}>\delta\}=\{g^{-1}h\mid \exp(-tk(g,h))>\delta\}=\left\{g^{-1}h\left\vert k(g,h)<\frac{-\log(\delta)}{t}\right.\right\}\]
  is precompact.

  On the other hand, suppose that $G$ satisfies the condition in theorem \ref{thm:CEkern}. We show that $G$ is coarsely embeddable. Fix an increasing sequence of compact sets $(K_n)_n$
  in $G$ such that $G=\bigcup_n \interior{K_n}$, where $\interior{K_n}$ denotes the interior of $K_n$. We can assume that $e\in K_1$.
  Take a sequence $(c_n)_n$ of positive real numbers that tends to infinity and let $(\varepsilon_n)_n$ be a sequence of strictly
  positive real numbers such that $\sum_n c_n\varepsilon_n<\infty$. We can assume that all the $\varepsilon_n<\frac12$. Then we find continuous positive definite kernels $(k_n)_n$ such that
  \begin{itemize}
  \item $\abs{k_n(g,h)-1}<\varepsilon_n$ whenever $g^{-1}h\in K_n$
  \item for every $\delta>0$, the set $\{g^{-1}h\in G\mid \abs{k_n(g,h)}>\delta\}$ is precompact.
  \end{itemize}

  Define a new kernel $k:G\times G\rightarrow\IR_+$ by the formula
  \[k(g,h)=\sum_n c_n\left(1-\frac{\abs{k_n(g,h)}^2}{k_n(g,g)k_n(h,h)}\right).\]
  An elementary computation shows that $k$ is conditionally negative definite. The series that defines $k$ is uniformly convergent on sets of the form
  $\{(g,h)\in G\times G\mid g^{-1}h\in L\}$ for compact sets $L\subset G$: let $L$ be compact, then $L\subset K_{N}$ for some $N\in \IN$.
  It follows that $\abs{1-\frac{\abs{k_n(g,h)}^2}{k_n(g,g)k_n(h,h)}}<8\varepsilon_n$ for all $g,h\in G$ with $g^{-1}h\in L$ and $n\geq N$. Therefore we also get that
  \[\sum_{n>N} c_n\left(1-\frac{\abs{k_n(g,h)}^2}{k_n(g,g)k_n(h,h)}\right)<8\sum_{n>N}c_n\varepsilon_n\rightarrow 0\text{ as }N\rightarrow\infty\]
  for all $g,h\in G$ with $g^{-1}h\in L$.
  As a consequence, we see that $k$ is well-defined, continuous and bounded on tubes.

  It remains to show that $k$ is proper. Let $R>0$ and take $n\in\IN$ such that $c_n>4R$. Then we see that $L=\{g^{-1}h\mid \abs{k_n(g,h)}>\frac{1}{4}\}$ is precompact.
  Whenever $g,h\in G$ are such that $g^{-1}h\not\in L$, we compute that
  \[k(g,h)\geq c_n\left(1-\frac{\abs{k_n(g,h)}^2}{k_n(g,g)k_n(h,h)}\right)\geq c_n\frac14>R.\]
\end{proof}

In fact, the continuity of the kernels $k:G\times G\rightarrow \IC$, the families $(\xi_g)_{g\in G}$ and the coarse embedding $u:G\rightarrow H$ is not important, neither for property A nor for uniform embeddability.
In the case of the family $(\xi_g)_g$ and the coarse embedding $u$, an argument was given in \cite{DL:A}. We give a short and elementary argument for the positive definite kernels $k$.
The main idea of the argument is that we restrict the non-continuous kernel to a metric lattice in $G$ and then extend it back to a continuous kernel on $G\times G$, using a partition of unity in $C_c(G)$.
A very similar argument was given in \cite{R05} to show that property A passes from a metric lattice to the
complete space.
\begin{proposition}
  A \lcsc group $G$ has property A (is coarsely embeddable) if and only if for every compact subset $K\subset G$ and every $\varepsilon>0$, there is a (not necessarily continuous) positive definite kernel $k:G\times G\rightarrow \IC$
  such that
  \begin{itemize}
  \item $\abs{k(g,h)-1}<\varepsilon$ whenever $g,h\in G$ satisfy $g^{-1}h\in K$
  \item (in the property A case) the set $\{g^{-1}h\mid k(g,h)\not=0\}$ is precompact
  \item (in the coarse embeddability case) for every $\delta>0$, the set $\{g^{-1}h\mid k(g,h)>\delta\}$ is precompact
  \end{itemize}
\end{proposition}
\begin{proof}
  It is clear that property A (respectively coarse embeddability) implies our condition. Suppose now that $G$ is a \lcsc group that satisfies our condition. Fix a compact neighborhood $U$ of identity in $G$,
  and take continuous functions $f_n:G\rightarrow [0,1]$, $(n\in\IN)$ with the following properties
  \begin{itemize}
  \item $\sum_nf_n(g)=1$ for all $g\in G$ and the convergence is uniform on compact subsets of $G$.
  \item for every $g\in G$, there are only finitely many $n\in\IN$ such that $f_n(g)\not=0$.
  \item for every $n\in\IN$, there is a group element $g_n\in G$ such that $\supp f_n\subset g_n U$.
  \end{itemize}

  Let $K\subset G$ be a compact subset, and let $\varepsilon>0$. By our condition, we find a positive definite kernel $k_0:G\times G\rightarrow \IC$
  with the properties
  \begin{itemize}
  \item $\abs{k_0(g,h)-1}<\varepsilon$ whenever $g,h\in G$ satisfy $g^{-1}h\in UKU^{-1}$
  \item (in the property A case) the set $\{g^{-1}h\mid \abs{k_0(g,h)}\not=0\}$ is precompact
  \item (in the coarse embeddability case) for every $\delta>0$, the set $\{g^{-1}h\mid \abs{k_0(g,h)}>\delta\}$ is precompact
  \end{itemize}
  Define a new kernel $k:G\times G\rightarrow \IC$ by the formula
  \[k(g,h)=\sum_{n,m}f_n(g)f_m(h)k_0(g_n,g_m)\text{ for all }g,h\in G.\]
  This sum converges uniformly on compact subsets of $G\times G$, so it follows that $k$ is continuous. It remains to show that it satisfies the following conditions
  \begin{enumerate}
  \item $k$ is positive definite
  \item $\abs{k(g,h)-1}<\varepsilon$ whenever $g,h\in G$ satisfy $g^{-1}h\in K$
  \item (in the property A case) the set $\{g^{-1}h\mid k(g,h)\not=0\}$ is precompact
  \item (in the coarse embeddability case) for every $\delta>0$, the set $\{g^{-1}h\mid k(g,h)>\delta\}$ is precompact.
  \end{enumerate}

  To prove (1), take $h_1,\ldots,h_s\in G$ and $c_1,\ldots, c_s\in\IC$ arbitrarily. Then we compute that
  \begin{align*}
    \sum_{i,j=1}^s \overline c_ic_jk(h_i,h_j)&=\sum_{n,m\in\IN}\sum_{i,j=1}^s \overline c_ic_jf_n(h_i)f_m(h_j)k_0(g_n,g_m)\\
    &=\sum_{n,m\in\IN} \left(\sum_i \overline c_if_n(h_i)\right) \left(\sum_i c_if_m(h_i)\right) k_0(g_n,g_m).
  \end{align*}
  Observe that only finitely many terms in this last sum are non-zero, so the result is positive because $k_0$ is positive definite.

  We prove condition (2) as follows. Let $g,h\in G$ be such that $g^{-1}h\in K$. If $n\in\IN$ is such that $f_n(g)\not=0$, then $g_n^{-1}g\in U$.
  Thus for every $n,m\in\IN$ with $f_n(g)f_m(h)\not=0$, we see that $g_n^{-1}g_m\in UKU^{-1}$, and hence $\abs{k_0(g_n,g_m)-1}<\varepsilon$. It follows that
  \begin{align*}
    \abs{k(g,h)-1}&= \abs{\sum_{n,m}f_n(g)f_m(h)k_0(g_n,g_m)-\sum_{n,m}f_n(g)f_m(h)}\\
    &\leq \sum_{n,m}f_n(g)f_m(h)\abs{k_0(g_n,g_m)-1}\\
    &<\sum_{n,m}f_n(g)f_m(h)\varepsilon\\
    &=\varepsilon.
  \end{align*}

  To prove (3) and (4) at once, we show that for every $\delta\geq 0$,
  \[\{g^{-1}h\mid \abs{k(g,h)}>\delta\}\subset U^{-1}\left\{\left.g^{-1}h\right\vert \abs{k_0(g,h)}>\delta\right\}U.\]
  The property A case corresponds to the case where $\delta=0$. Let $g,h\in G$ be such that $\abs{k(g,h)}>\delta$. Then there
  is at least one pair $n,m\in\IN$ such that $f_n(g)f_m(h)\not=0$ while $\abs{k_0(g_n,g_m)}>\delta$. But then we get that
  \[g^{-1}h=g^{-1}g_n\, g_n^{-1}g_m\,g_m^{-1}h\in U^{-1}\left\{\left.g_0^{-1}h_0\right\vert \abs{k_0(g_0,h_0)}>\delta\right\}U.\]
\end{proof}

\section{Proper cocycles}
\label{sect:propercocycle}
In this section, we introduce proper cocycles.
Proper cocycles were first introduced by Jolissaint in \cite{Jolissaint:propercocycle}. Our definition is inspired by his notion of proper cocycles,
but we use a slightly weaker version. We did this because we could not show that measure correspondences give rise to proper cocycles
in Jolisaint's sense, but they do give rise to proper cocycles in our sense. Jolissaint's results based on proper cocycles remain valid
for our proper cocycles, as we show in section \ref{sect:proofs}.

\begin{definition}
  Let $G\actson (X,\mu)$ be a non-singular Borel action of a \lcsc group $G$ on a standard probability space, and let $H$ be another \lcsc group.
  \begin{itemize}
  \item A Borel map $\omega:G\times X\rightarrow H$ is called a \emph{cocycle} if for every $g,h\in G$, the relation
    \[\omega(gh,x)=\omega(g,hx)\omega(h,x)\text{ holds for almost every }x\in X.\]
  \item Two cocycles $\omega_1,\omega_2:G\times X\rightarrow H$ are cohomologous if there is a Borel map $\varphi:X\rightarrow H$
    such that for every $g\in G$ separately, we have that
    \[\omega_2(g,x)=\varphi(gx)^{-1}\omega_1(g,x)\varphi(x)\text{ for almost every }x\in X.\]
  \end{itemize}
\end{definition}
Whenever we say that $\omega$ is a cocycle, we mean that $\omega$ is a Borel cocycle. Similarly, when we say that $G\actson (X,\mu)$
is a non-singular action, we mean that the action is Borel and that $(X,\mu)$ is a standard probability space.

\begin{definition}[see also \cite{Jolissaint:propercocycle}]
  \label{def:propercocycle}
  Let $G,H$ be \lcsc groups and let $G\actson (X,\mu)$ be a non-singular action. 
  \begin{itemize}
  \item A cocycle $\omega:G\times X\rightarrow H$ is said
    to be \emph{proper with respect to a family $\cA$ of Borel sets in $X$} if
    \begin{enumerate}
    \item for every compact subset $K\subset G$ and every $A,B\in\cA$, we find a precompact set
      $L(K,A,B)\subset H$ such that, for every $g\in K$ we get that
      \[\omega(g,x)\in L(K,A,B) \text{ for almost all }x\in A\cap g^{-1}B\]
    \item for every compact subset $L\subset H$ and every $A,B\in\cA$, we get that the set $K(L,A,B)$
      of all $g\in G$ such that
      \[\mu\{x\in X\mid x\in A, gx\in B, \omega(g,x)\in L\}>0,\]
      is precompact in $G$.
    \end{enumerate}
  \item A family $\cA$ of Borel sets in $X$ is said to be \emph{large} if it is closed under finite unions and under taking Borel subsets,
    and if for every $\varepsilon>0$ there is a set $A\in \cA$ such that $\mu(X\setminus A)<\varepsilon$.
  \item A cocycle $\omega:G\times X\rightarrow H$ is said to be \emph{proper} if $\omega$ is proper with respect to some large family $\cA$.
  \end{itemize}
\end{definition}

Throughout this section, we will use the following examples of large families.
\begin{observation}
  \label{obs:union}
  Let $(X,\mu)$ be a standard probability space.
  \begin{itemize}
    \item If $\varphi:X\rightarrow H$ is a Borel map into a $\sigma$-compact space, then the family
      \[\cA=\{A\subset X\mid A \text{ is Borel and }\varphi(A)\text{ is precompact}\}\]
      is a large family.
    \item If $(\cA_n)_{n\in\IN}$ is a countable sequence of large families in $X$,
      then the intersection\\
      $\cA=\bigcap_{n\in\IN}\cA_n$ is still large.
    \end{itemize}
\end{observation}
\begin{proof}
  For the first point, observe that $\cA$ is clearly closed under finite unions and under taking Borel subsets. We can write $H$ as a countable union
  $H=\bigcup_n L_n$ of compact sets. So $X$ is the countable union $X=\bigcup_n A_n$, where $A_n=\varphi^{-1}(L_n)\in\cA$ for all $n$.
  Since $\mu$ is a probability measure, we get that the measure $\mu(X\setminus A_n)$ tends to $0$.

  For the second point, it is clear that $\cA$ is closed under finite unions and under taking Borel subsets. Moreover, because the $\cA_n$
  are closed under taking Borel subsets, we see that
  \[\cA=\bigcap_n \cA_n=\left\{\left.\bigcap_n A_n\right\vert A_n\in\cA_n\right\}.\]
  We only have to show that, for every $\varepsilon>0$,
  there is a set $A\in\cA$ such that $\mu(X\setminus A)<\varepsilon$. Let $\varepsilon>0$. For every $n\in\IN$, take a set $A_n\in\cA_n$
  with measure $\mu(X\setminus A_n)<\frac{1}{2^n}\varepsilon$. Observe that $A=\bigcap_{n=1}^\infty A_n\in \cA$ and that
  \[\mu(X\setminus A)\leq \sum_{n=1}^\infty \mu(X\setminus A_n)<\sum_{n=1}^\infty \frac{1}{2^n}\varepsilon=\varepsilon.\]
\end{proof}

We want to compare our notion of proper cocycle with Jolissaint's notion. In order to do that, we call hist type of proper cocycles ``Jolissaint-proper''.
\begin{definition}[\cite{Jolissaint:propercocycle}]
  \label{def:J-proper}
  Let $G,H$ be \lcsc groups and let $G\actson (X,\mu)$ be a non-singular action. We say that a cocycle $\omega:G\times X\rightarrow H$ is Jolissaint-proper
  if there is a family $\cA$ such that
  \begin{enumerate}
  \item for every $A\in\cA$ and for every compact set $K\subset G$, the set $\omega(K\times A)\subset H$ is precompact.
  \item for every compact set $L\subset H$ and every $A\in\cA$, we get that the set $K_J(L,A)$ of all $g\in G$ such that
    \[\mu\{x\in X\mid x\in A, gx\in A, \omega(g,x)\in L\}>0,\]
    is precompact in $G$.
  \item for every $\varepsilon>0$, there is a set $A\in\cA$ such that $\mu(X\setminus A)<\varepsilon$.
  \end{enumerate}
\end{definition}

Our first observation is that Jolissaint's notion of proper cocycle is not invariant under cohomology, while ours is.
\begin{example}
  \label{ex:non-J}
  Consider the action $\IR\actson \Circle^1=\IR/\IZ$, and the cocycle $\omega:\IR\times\Circle^1\rightarrow \IR$ that is defined by
  $\omega(g,x)=g$. This cocycle is clearly Jolissaint-proper, where we can take $\cA$ to be the family of all Borel subsets of $\Circle^1$.
  Take now any unbounded Borel function $\varphi:\Circle^1\rightarrow\IR$. Then the formula $\omega_1(g,x)=\varphi(gx)^{-1}\omega(g,x)\varphi(x)$
  defines a new cocycle that is cohomologous to $\omega$. But $\omega_1$ is not Jolissaint-proper, because for every fixed $x\in X$, we get that
  $gx$ runs over all of $\Circle^1$ when $g\in K=[0,1]$. So $\omega_1(K\times A)$ is not precompact for any non-empty set $A$.
  This shows that $\omega_1$ can never satisfy condition (1) from definition \ref{def:J-proper}.
\end{example}

\begin{observation}
  \label{obs:cohom}
  Let $G,H$ be \lcsc groups and let $G\actson (X,\mu)$ be a non-singular action.
  Suppose that $\omega_1,\omega_2:G\times X\rightarrow H$ are two cocycles that are cohomologous.
  If $\omega_1$ is proper, then $\omega_2$ is also proper.
\end{observation}
\begin{proof}
  Let $\varphi:X\rightarrow H$ be a Borel function such that for all $g\in G$, we have that
  $\omega_2(g,x)=\varphi(gx)\omega_1(g,x)\varphi(x)^{-1}$ for almost all $x\in X$.
  Suppose that $\omega_1$ is proper with respect to the large family $\cA$. Consider the large family $\cA_1$
  of all Borel sets $A\subset X$ such that $\varphi(A)$ is a precompact set in $H$. Then it is clear that $\omega_2$ is proper with respect to
  the large family $\cA\cap\cA_1$.
\end{proof}

If a cocycle $\omega$ is proper in our sense, with respect to a family $\cA$, it is clear that $\omega$ is also proper with respect
to the larger family $\overline{\cA}$ that consists of all Borel subsets of finite unions of elements in $\cA$. So we can always assume
that $\omega$ is a proper cocycle with respect to a family that is closed under finite unions and under taking Borel subsets. This last
fact is also true for Jolissaint-properness, though it is a little bit less obvious.
\begin{proposition}
  \label{prop:J-our}
  Let $G,H$ be \lcsc groups and let $G\actson (X,\mu)$ be a non-singular action. If $\omega$ is Jolissaint-proper with respect to
  a family $\cA$, then $\omega$ is also Jolissaint-proper with respect to a large family $\overline{\cA}$.
  In particular, it is a proper cocycle in our sense.
\end{proposition}
\begin{proof}
  The only thing that is not clear is that we can assume that $\cA$ is closed under taking finite unions.
  Since $X$ is the countable union of sets in $\cA$ (up to measure 0), we find a sequence of sets $A_n\in\cA$ such
  that $X$ is the countable disjoint union of their saturations, i.e.\@ $X=\sqcup_n GA_n$, up to measure 0. The family $\overline{\cA}$
  is now the set of all Borel subsets of sets of the form $K_1A_1\sqcup\ldots\sqcup K_nA_n$ where $n\in\IN$ and $K_1,\ldots,K_n\subset G$
  are compact. It is clear that $\overline{\cA}$ is a large family. We show that $\omega$ is Jolissaint-proper with respect to $\overline{\cA}$.
  For the first condition, let $K\subset G$ be compact and let $A\subset K_1A_1\sqcup\ldots\sqcup K_nA_n\in\cA$. Then we see that
  \[\omega(K\times A)\subset \bigcup_{k=1}^n\omega(K\times(K_kA_k))\subset \bigcup_{k=1}^n \omega(KK_k\times A_k)\omega(K_k\times A_k)^{-1},\]
  and this last set is precompact.

  For the second condition, let $L\subset H$ be a compact set, and take a set $A\subset\bigsqcup_{k=1}^nK_nA_n$ in $\cA$.
  Consider the precompact sets $L_k=\omega(K_k\times A_k)^{-1}L\,\omega(K_k\times A_k)$ in $H$ and define
  $K=\bigcup_{k=1}^n K_kK_J(L_k,A_k)K_k^{-1}$. Observe that $K$ is precompact. Let $g\in G$ be such that
  \[\mu\{x\in X\mid x\in A, gx\in A, \omega(g,x)\in L\}>0.\]
  We show that $g\in K$. For every $x\in X$ with $x\in A$, $gx\in A$ and $\omega(g,x)\in L$, we find
  $k,l\leq n$ and $h_1\in K_k, h_2\in K_l$ such that $h_1^{-1}x\in A_k$ and $h_2^{-1}gx\in A_l$. Because the
  saturations of $A_k,A_l$ are disjoint for different $k,l$, we see that $k=l$. So $h_2^{-1}gh_1$ is an element in $G$
  for which the measure
  \[\mu\{x\in X\mid x\in A_k, h_2^{-1}gh_1x\in A_k, \omega(h_2^{-1}gh_1,x)\in L_k\}>0.\]
  So $g\in K_kK_J(L_k,A_k)K_k^{-1}\subset K$.
\end{proof}

\section{Measure correspondences}
In this section, we introduce measure correspondences between arbitrary \lcsc groups. This is used to define measure equivalence
between \lcsc groups. We show that this more general notion coincides with the classical one for unimodular groups.
A good exposition of the unimodular case is given in \cite[Appendix A]{FurmanBaderSauer:MErigidity}. We follow a similar strategy in the general case.

\begin{definition}
  \label{def:MC}
  Let $G,H$ be two \lcsc groups and let $G\times H\actson (\Omega,\eta)$ be a non-singular action. In the following
  we consider the Haar measure on $G,H$.
  We say that $\Omega$ is a measure $G$-$H$-correspondence if there exist standard probability spaces $(X,\mu)$ and $(Y,\nu)$ and (almost everywhere defined) measure class preserving Borel isomorphisms
  $\varphi:X\times H\rightarrow \Omega$ and $\psi:Y\times G\rightarrow \Omega$ such that for all $h\in H$ we have that $\varphi(x,hk)=h\varphi(x,k)$ for almost all $(x,k)\in X\times H$,
  and such that for every $g\in G$, we have that $\psi(y,gk)=g\psi(y,k)$ for almost all $(y,k)\in Y\times G$.
\end{definition}

Standard examples of measure correspondences are the following:
\begin{itemize}
\item For every \lcsc group $G$ with Haar measure, we have the \emph{identity $G$-$G$-correspondence} $\Omega=G$
  with the left-right action of $G\times G$.
\item When $H_1,H_2\subset G$ are closed subgroups, then we find an $H_1$-$H_2$-correspondence $\Omega=G$ with the left action of $H_1$
  and the right action of $H_2$.
\item For every two \lcsc groups $G,H$ with Haar measure, we define the \emph{coarse $G$-$H$-correspondence} $\Omega=G\times H$
  with the left translation action of $G\times H$.
\item A non-singular action $G\actson (X,\mu)$, induces a $G$-$G$-correspondence $\Omega=X\times G$ with
  an action of $G\times G$ that is defined by $(g,h)\cdot (x,k)=(gx,gkh^{-1})$.
\end{itemize}

In the rest of this section, all Borel maps are almost everywhere defined, and on all the \lcsc groups we consider a probability measure
that is equivalent to the Haar measure. We only consider equality almost everywhere.
\begin{observation}
  \label{obs:comm}
  Let $H$ be a \lcsc group and let $(X,\mu)$ be a standard probability space.
  Let $\alpha_1:X\rightarrow X$ be a measure class preserving Borel isomorphism, and let $\alpha_2:X\rightarrow H$
  be a Borel map. Then the formula
  \[\alpha(x,h)=\left(\alpha_1(x),h\alpha_2(x)^{-1}\right)\text{ for almost all }(x,h)\in X\times H\]
  defines a measure class preserving Borel isomorphism of $X\times H$ that commutes with the action of $H$.
  Moreover, all measure class preserving Borel isomorphisms of $X\times H$ that commute with the action of $H$
  are of this form.
\end{observation}
\begin{proof}
  It is clear that every $\alpha$ of this form is a measure class preserving Borel isomorphism that commutes with the action of $H$.

  When $\alpha$ is a measure class preserving Borel isomorphism, then
  we find Borel maps\\ \hbox{$\alpha_1:X\times H\rightarrow X$} and $\alpha_2:X\times H\rightarrow H$ such that $\alpha(x,h)=(\alpha_1(x,h),h\alpha_2(x,h)^{-1})$ for all $(x,h)\in X\times H$.
  When $\alpha$ commutes with the action of $H$, these maps $\alpha_1,\alpha_2$ are invariant under the action of $H$, so they essentially depend only on $x$. The map $\alpha_1$
  is a measure class preserving Borel isomorphism because $\alpha$ is.
\end{proof}

Suppose that $\Omega$ is a measure correspondence between two \lcsc groups $G,H$. Then there exists a measure class preserving Borel isomorphism $\varphi:X\times H\rightarrow \Omega$
that commutes with the action of $H$. By observation \ref{obs:comm}, we find a non-singular action $G\actson X$
and a Borel cocycle $\omega:G\times X\rightarrow H$ such that
\[g\varphi(x,h)=\varphi\left(gx,h\omega(g,x)^{-1}\right)\text{ for almost all }(x,h)\in X\times H.\]
This action and cocycle are completely determined by the map $\varphi$.
We say that two non-singular actions $G\actson (X,\mu)$ and $G\actson (X^\prime,\mu^\prime)$ are isomorphic if there
is a measure class preserving Borel isomorphism $\Delta:X\rightarrow X^\prime$ that commutes with the action of $G$.
Choosing a different Borel isomorphism $\varphi$
yields an isomorphic action of $G$ and a cohomologous cocycle $\omega$. So the action and cocycle are well-defined by $\Omega$ up to isomorphism
and cohomology. Similarly, the Borel isomorphism $\psi:Y\times G\rightarrow \Omega$ yields a non-singular action $H\actson Y$ and a cocycle
$\beta:H\times Y\rightarrow G$.
We say that the actions $G\actson X$, $H\actson Y$ and the cocycles $\omega:G\times X\rightarrow H$ and $\beta:H\times Y\rightarrow G$
are associated to $\Omega$.

\def\niets{
These two cocycles are each others inverse in the sense of the following lemma.
\begin{lemma}
  There exist Borel maps $s:X\rightarrow G$ and $u:X\rightarrow Y$ such that for all $g\in G$ we have that
  \begin{align*}
    u(gx)&=\omega(g,x)\text{ for almost all }x\in X
    g&=s(gx)^{-1}\beta(\omega(g,x),u(x))s(x)\text{ for almost all }x\in X.
  \end{align*}
  Moreover, for every set $V\subset Y$ with measure $1$, there is $h\in H$ such that $u^{-1}(h^{-1}V)$ has measure 1.
\end{lemma}
\begin{proof}
  By Mackey's cocycle theorem, we can assume that $\omega,\beta$ are strict cocycles, i.e.\@ they satisfy the relations in definition \ref{def:cocycles}
  for all $x\in X$ instead of only almost all $x\in X$.
  Fix measure class preserving
  Borel isomorphisms $\varphi:X\times H\rightarrow \Omega$ and $\psi:Y\times G\rightarrow \Omega$ as in observation \ref{obs:cocycles}.
  Because $\psi$ commutes with the $H$-action, we find unique Borel maps $s:X\rightarrow G$ and $u:X\rightarrow Y$
  such that
  \[\psi(x,h)=h\varphi(u(x),s(x))\text{ for almost all }(x,h)\in X\times H.\]

  For every $g\in G$, we compute that for almost all $(x,h)\in X\times H$,
  \begin{align*}
    g\psi(x,h)&=\psi(gx,h\omega(g,x)^{-1})\\
    \Vert\qquad&\qquad\Vert\\
    h\varphi(u(x),gs(x))&=h\omega(g,x)^{-1}\varphi(u(gx),s(gx))\\
    \Vert\qquad&\qquad\Vert\\
    \varphi(hu(x),gs(x)\beta(h,u(x))^{-1})&=\varphi(h\omega(g,x)^{-1}u(gx),s(gx)\beta(h\omega(g,x)^{-1},u(gx))^{-1})
  \end{align*}
  It follows that for all $g\in G$ we have that for almost all $(x,h)\in X\times H$
  \begin{align*}
    u(gx)&=\omega(g,x)u(x)\text{ for almost all }x\in X\\
    gs(x)\beta(h,u(x))^{-1}&=s(gx)\beta(\omega(g,x)^{-1},u(gx))^{-1}\beta(h,\omega(g,x)^{-1}u(gx))^{-1}\\
    &=s(gx)\beta(\omega(g,x),u(x))\beta(h,u(x))^{-1}
  \end{align*}
\end{proof}}

There are two important operations on measure correspondences: composition and the opposite.
\begin{definition}
  Let $G,H,K$ be \lcsc groups.
  \begin{itemize}
  \item If $\Omega$ is a measure correspondence between $G,H$, then the opposite measure correspondence $\overline{\Omega}$
    between $H$ and $G$ is $\overline{\Omega}=\Omega$ with the obvious action of $H\times G$.
  \item If $\Omega_1$ and $\Omega_2$ are measure correspondences between $G,H$ and $H,K$ respectively. Then we define
    $\Omega_1\otimes_H\Omega_2$ to be the quotient of $\Omega_1\times\Omega_2$ by the action of $H$ that is given by $h\cdot(x,y)=(hx,hy)$.
    On $\Omega_1\otimes_H\Omega_2$ we consider the probability measure that is induced by the quotient map from $\Omega_1\times\Omega_2$,
    and together with the action of $G\times K$ that is given by $(g,k)\cdot(x,y)=(gx,ky)$. Proposition \ref{prop:comp} below shows that $\Omega_1\otimes_H\Omega_2$
    is a measure correspondence between $G,K$.
  \end{itemize}
\end{definition}

\begin{proposition}
  \label{prop:comp}
  Let $G,H,K$ be \lcsc groups and let $(\Omega_1,\eta_1)$ and $(\Omega_2,\eta_2)$ be measure correspondences between $G,H$ and $H,K$ respectively.
  Let $\varphi_1:X_1\times H\rightarrow \Omega_1$, $\psi_1:Y_1\times G\rightarrow \Omega_2$, $\varphi_2:X_2\times K\rightarrow\Omega_2$ and $\psi_2:Y_2\times H\rightarrow \Omega_2$
  be measure class preserving Borel isomorphisms as in definition \ref{def:MC}.
  Consider the quotient map $\pi:\Omega_1\times\Omega_2\rightarrow \Omega_1\otimes_H\Omega_2=\Omega$.
  Then the following maps are measure class preserving Borel isomorphisms.
  \begin{align*}
    \varphi:X_1\times X_2\times K\rightarrow \Omega:&\varphi(x_1,x_2,k)=\pi(\varphi_1(x_1,e),\varphi_2(x_2,k))\\
    \psi:Y_1\times Y_2\times G\rightarrow \Omega:&\psi(y_1,y_2,g)=\pi(\psi_1(y_1,g),\psi_2(y_2,e))
  \end{align*}
  The map $\varphi$ clearly commutes with action of $K$ and $\psi$ commutes with the action of $G$. This shows that $\Omega_1\otimes_H\Omega_2$ is a measure correspondence.
\end{proposition}
\begin{proof}
  By definition, it is clear that $\varphi,\psi$ are Borel maps such that $\varphi^{-1}(E)$ has measure $0$ for every set $E$ of measure $0$, and similar for $\psi$.
  So it suffices to give an inverse Borel map $\varphi^\prime$ for $\varphi$ and $\psi^\prime$ for $\psi$.
  If we write $\varphi_1^{-1}:\Omega_1\rightarrow X_1\times H$ as $\varphi_1^{-1}(x)=(u(x),s(x))$ for almost all $x\in\Omega_2$, then we can define
  $\varphi^\prime_0:\Omega_1\times\Omega_2\rightarrow X_1\times X_2\times K$ by
  \[\varphi^\prime_0(x,y)=(u(x),\varphi_2^{-1}(s(x)^{-1}y)).\]
  This function $\varphi^\prime_0$ is invariant under the action of $H$ and hence defines a map $\varphi^\prime:\Omega_1\otimes_H\Omega_2\rightarrow X_1\times X_2\times K$.
  An elementary computation shows that $\varphi^\prime$ is the inverse of $\varphi$.
  The same idea works for $\psi$.
\end{proof}

\begin{definition}
  \label{def:ME}
  Let $(\Omega,\eta)$ be a measure correspondence between \lcsc groups $G,H$. Consider the non-singular actions $G\actson (X,\mu)$ and $H\actson (Y,\nu)$
  associated to $\Omega$. We say that $G\actson (X,\mu)$ has an invariant probability measure if there exists a $G$-invariant probability measure on $X$ that is equivalent with $\mu$.
  \begin{itemize}
  \item We say that $\Omega$ is a measure equivalence coupling if both $G\actson (X,\mu)$ and $H\actson (Y,\nu)$ have invariant probability measures. In that case we say that $G$ and $H$ are measure equivalent.
  \item We say that $\Omega$ is a measure equivalence subgroup coupling if $G\actson (X,\mu)$ has an invariant probability measure. In that case we say that $G$ is a measure equivalence subgroup of $H$.
  \end{itemize}
\end{definition}

It is now easy to see that the composition of two measure equivalence (subgroup) couplings is still a measure equivalence (subgroup) coupling.

We show that definition \ref{def:ME} is equivalent to the classical one for unimodular groups.
\begin{proposition}
  \label{thm:MEdef}
  Let $G,H$ be unimodular \lcsc groups and let $(\Omega,\eta)$ be a measure equivalence coupling.
  Consider standard probability spaces $(X,\mu)$ and $(Y,\nu)$ together with measure class preserving
  Borel isomorphisms $\varphi:X\times H\rightarrow\Omega$ and $\psi:Y\times G\rightarrow\Omega$ as in definition \ref{def:MC}.
  Then the following measures exist:
  \begin{itemize}
  \item An infinite $G\times H$-invariant measure $\eta^{\prime}$ on $\Omega$, that is equivalent to $\eta$.
  \item Finite measures $\mu^\prime,\nu^\prime$ on $X,Y$ that are equivalent to $\mu,\nu$ and such that
    the Borel isomorphisms
    \begin{align*}
      \varphi&:(X,\mu^\prime)\times (H,\mu_H)\rightarrow (\Omega,\eta^\prime)\\
      \psi&:(Y,\nu^\prime)\times (G,\mu_G)\rightarrow (\Omega,\eta^\prime)
    \end{align*}
    are measure preserving. The measures $\mu_G$, $\mu_H$ are Haar measures on $G$, $H$.
  \end{itemize}
  In particular, $(\Omega,\eta^\prime)$ is a measure equivalence coupling in the sense of \cite[definition 1.1]{FurmanBaderSauer:MErigidity}.
\end{proposition}
\begin{proof}
  Consider the actions $G\actson (X,\mu)$, $H\actson (Y,\nu)$ and cocycles $\omega,\beta$ associated to $\Omega$.
  Since $\Omega$ is a measure equivalence coupling, we can assume that $\mu$ and $\nu$ are invariant
  probability measures on $X$ and $Y$.
  Because $H$ is unimodular, we see that the action $G\times H\actson (X\times H)$ defined by
  $(g,h)\cdot(x,k)=(gx,hk\omega(g,x)^{-1})$
  preserves the measure $\mu\times \mu_H$ where $\mu_H$ denotes the Haar measure on $H$.
  The push-forward measure $\eta_1=\varphi_\ast(\mu\times\mu_H)$
  is preserved by the $G\times H$ action on $\Omega$.
  Similarly, $\eta_2=\psi_{\ast}(\nu\times\mu_G)$ is a $G\times H$-invariant measure on $\Omega$
  that is equivalent to $\eta_1$. So we find a measurable $G\times H$-invariant function
  $f:\Omega\rightarrow (0,\infty)$ such that $d\eta_2(x)=f(x)d\eta_1(x)$. Consider the $G\times H$-invariant
  measure $\eta^\prime$ on $\Omega$ that is defined by $d\eta^\prime(x)=\min(1,f(x))d\eta_1(x)$.
  Then $\eta^\prime$ is still $G\times H$-invariant and equivalent to $\eta$, and $\eta^\prime$ is smaller than both $\eta_1$ and $\eta_2$.

  The measure $\varphi_\ast^{-1}\eta^\prime$ is an $H$-invariant measure on $X\times H$, so it is of the form
  $\mu^\prime\times \mu_H$ with $\mu^\prime$ equivalent to $\mu$. Moreover, $\mu^\prime$ is smaller that $\mu$
  so it is still a finite measure, and $\varphi:(X,\mu^\prime)\times (H,\mu_H)\rightarrow (\Omega,\eta^\prime)$
  is measure preserving. Similarly we find a finite measure $\nu^\prime$ such that
  $\psi:(Y,\nu^\prime)\times (G,\mu_G)\rightarrow(\Omega,\eta^\prime)$ is measure preserving.
\end{proof}

\section{Measure correspondences give rise to proper cocycles}
In this section, we show that measure correspondences give rise to proper cocycles.
For discrete groups, this is relatively easy, see \cite{Jolissaint:permanence}. We give a complete proof for the general case.
\begin{theorem}
  \label{thm:MC}
  Let $G,H$ be \lcsc groups and let $(\Omega,\eta)$ be a measure correspondence between $G$ and $H$.
  Consider the non-singular actions $G\actson (X,\mu)$
  and $H\actson (Y,\nu)$ associated with $\Omega$, together with the cocycles
  $\omega:G\times X\rightarrow H$ and $\beta:H\times Y\rightarrow G$.
  Then both $\omega$ and $\beta$ are proper cocycles.
\end{theorem}

\begin{proof}
  By symmetry, we only have to show that $\omega$ is a proper cocycle. 
  By results of Mackey, we can assume
  that the actions $G\actson (X,\mu)$ and $H\actson (Y,\nu)$ are everywhere defined Borel actions \cite{Mackey:actions},
  and by the Mackey cocycle theorem, we can assume that $\omega$ and $\beta$ are 
  strict Borel cocycles. This means that the cocycle relation
  $\omega(gh,x)=\omega(g,hx)\omega(h,x)$
  holds for all $g,h\in G$ and all $x\in X$ instead of almost all $x\in X$, and similarly for $\beta$.


  We proceed in three steps.
  \begin{step}
    For every compact set $K\subset G$, there is an increasing sequence $(A_n)_n$ of Borel sets in $X$
    with $X=\bigcup_n A_n$ and such that for all $n\in\IN$, the set
    \[\{\omega(g,x)\mid x\in A_n, gx\in A_n, g\in K\}\]
    is precompact.
  \end{step}
  Take an increasing sequence of compact sets $L_n\subset H$ such that $H=\bigcup_n L_n$.
  Denote $U_n=\omega^{-1}(L_n)\subset G\times X$. For every $x\in X$ we denote
  $U_{n,x}=\{g\in G\mid (g,x)\in U_n\}$.

  Consider the left Haar measure on $G$ and
  fix a strictly positive Borel function $f:G\rightarrow \IR$, with $\norm{f}_2=1$.
  For every $n\in\IN$ and $x\in X$, we write $f_{n,x}=\chara_{U_{n,x}}f$.
  It follows that, for every $x\in X$, we get that
  \[\norm{f_{n,x}-f}_2\rightarrow 0\text{ when }n\rightarrow\infty.\]

  For two functions $f_1,f_2\in\Lp^2(G)$, we define a function $f_1\star f_2:G\rightarrow\IR$ by
  the formula
  \[(f_1\star f_2)(g)=\int_G f_1(hg^{-1})f_2(h) dh.\]
  It is easy to see that $\star$ is bilinear and that $\abs{(f_1\star f_2)(g)}\leq \Delta(g)\norm{f_1}_2\norm{f_2}_2$
  for all $g\in G$, where $\Delta_G$ is the modular function of $G$.
  Using a similar argument as in \cite[theorem 8.8]{Folland}, it is easy to see that $f_1\star f_2$
  is a continuous function.
  
  Consider the functions $\tilde f=f\star f$ and $\tilde f_{n,x,y}=f_{n,y}\star f_{n,x}$.
  By the above, these functions are continuous and positive. Moreover, $\tilde f$ is strictly positive.
  Hence $\varepsilon=\min\{\tilde f(g)\Delta_G(g)^{-1}\mid g\in K\}$ is strictly positive. Moreover, we see that
  \[\abs{(\tilde f-\tilde f_{n,x,y})(g)}\leq \Delta_G(g)(\norm{f-f_{n,y}}_2+\norm{f-f_{n,x}}_2),\]
  for all $n\in\IN$ and $x,y\in X$.

  For every $n\in\IN$, we define an increasing sequence of Borel sets $A_n$ in $X$ by the relation
  \[A_n=\left\{x\in X\left\vert \norm{f_{n,x}-f}_2<\frac{1}{3}\varepsilon\right.\right\}.\]
  It is clear that $X=\bigcup_n A_n$. Fix $n\in\IN$ and suppose that $x\in A_n$, $gx\in A_n$
  and $g\in K$. Then we see that $\tilde f_{n,x,gx}(g)\geq \frac13 \tilde f(g)>0$. In particular, there is an $h\in U_{n,x}$
  with $k=hg^{-1}\in U_{n,gx}$, so
  \[
  \omega(g,x)=\omega(k^{-1}h,x)=\omega(k,gx)^{-1}\omega(h,x)\in L_n^{-1}L_n.
  \]
  This last set is precompact.

  \begin{step}
    There exist Borel maps $s:X\rightarrow G$ and $u:X\rightarrow Y$ such that for all $g\in G$ we have that
    \begin{align*}
      u(gx)&=\omega(g,x)u(x)\text{ for almost all }x\in X\\
      g&=s(gx)\beta(\omega(g,x),u(x))s(x)^{-1}\text{ for almost all }x\in X.
    \end{align*}
  \end{step}
  Fix measure class preserving
  Borel isomorphisms $\varphi:X\times H\rightarrow \Omega$ and $\psi:Y\times G\rightarrow \Omega$ as in definition \ref{def:MC}.
  We find Borel maps $u_0:X\times H\rightarrow Y$ and $s_0:X\times H\rightarrow G$ such that
  \[\varphi(x,h)=\psi\left(h\,u_0(x,h),s_0(x,h)\beta(h,u_0(x,h))^{-1}\right),\]
  for almost all $(x,h)\in X\times H$. For every $k\in H$, we see that
  \begin{align*}
    \varphi(x,kh)&=k\varphi(x,h)\\
    &=\psi\left(kh\,u_0(x,h),s_0(x,h)\beta(kh,u_0(x,h))^{-1}\right)\\
    \text{ and }\varphi(x,kh)&=\psi\left(kh\,u_0(x,kh),s_0(x,kh)\beta(kh,u_0(x,kh))^{-1}\right)
  \end{align*}
  for almost all $(x,h)\in X\times H$. It follows that $u_0$ and $s_0$ are invariant under the action of $H$
  and hence there are Borel functions $u:X\rightarrow Y$ and $s:X\rightarrow G$ such that $u_0(x,h)=u(x)$ and $s_0(x,h)=s(x)$
  almost everywhere. These maps $u,s$ satisfy
  \[\varphi(x,h)=\psi(h\,u(x),s(x)\beta(h,u(x))^{-1})\]
  for almost all $(x,h)\in X\times H$.

  For every $g\in G$, we compute that for almost all $(x,h)\in X\times H$,
  \begin{align*}
    g\varphi(x,h)&=\varphi(gx,h\omega(g,x)^{-1})\\
    &=\psi\left(h\omega(g,x)^{-1}u(gx),s(gx)\beta(h\omega(g,x)^{-1},u(gx))^{-1}\right)\\
    \text{ and }g\varphi(x,h)&=\psi\left(h\,u(x),g\,s(x)\beta(h,u(x))^{-1}\right)\\
  \end{align*}
  It follows that for all $g\in G$ we have that
  \begin{align*}
    u(gx)&=\omega(g,x)u(x)\text{ for almost all }x\in X\\
    gs(x)&=s(gx)\beta(\omega(g,x)^{-1},u(gx))^{-1}\\
  \end{align*}
  This finishes the proof of step 2.

  \begin{step}
    The cocycle $\omega$ is a proper cocycle.
  \end{step}
  Fix an increasing sequence $(K_n)_n$ of precompact open subsets of $G$ such that $G=\bigcup_n K_n$.
  For every $n\in\IN$, step 1 gives us an increasing sequence $A_{n,k}$ of Borel sets in $X$ with $X=\bigcup_k A_{n,k}$
  and such that for all $n,k\in\IN$, we get that the set
  \[\{\omega(g,x)\mid x,gx\in A_{n,k}, g\in K_n\}\]
  is precompact.

  Similarly, for an increasing sequence $(L_n)_n$ of precompact open subsets of $H$ with $H=\bigcup_n L_n$,
  we find similar increasing sequences of Borel sets $B_{n,k}$ in $Y$ such that $\bigcup_k B_{n,k}=Y$
  and for all $n,k\in\IN$, we get that
  \[\{\beta(h,y)\mid y,hy\in B_{n,k}, h\in L_n\}\]
  is precompact.

  By step 2, we find Borel maps $s:X\rightarrow G$ and $u:X\rightarrow Y$ such that for all $g\in G$
  we have that
  \begin{align*}
    u(gx)&=\omega(g,x)u(x)\text{ for almost all }x\in X\\
    g&=s(gx)\beta(\omega(g,x),u(x))s(x)^{-1}\text{ for almost all }x\in X.
  \end{align*}

  Consider the family $\cA$ of all Borel sets $A\subset X$ such that $s(A)$ is precompact and such that
  for every $n\in\IN$ there exists $k\in \IN$ such that $A\subset A_{n,k}$ and $u(A)\subset B_{n,k}$.
  By observation \ref{obs:union}, this is a large family. We show that $\omega$ is a proper cocycle with
  respect to $\cA$.

  To show condition (1) from definition \ref{def:propercocycle}, take a compact set $K\subset G$ and two sets $A,B\in\cA$.
  Then there is an $n\in\IN$ such that $K\subset K_n$. Moreover, there is a $k\in\IN$ such that $A,B\subset A_{n,k}$.
  Now we see that
  \[\{\omega(g,x)\mid x\in A, gx\in B, g\in K\}\subset \{\omega(g,x)\mid x,gx\in A_{n,k}, g\in K_n\},\]
  and the second set was assumed to be precompact.

  To show condition (2), let $L\subset H$ be a compact set and take $A,B\in\cA$.
  Then there is an $n\in\IN$ such that $L\subset L_n$, and there is a $k\in\IN$ with $u(A),u(B)\subset B_{n,k}$.
  Moreover, $s(A)$ and $s(B)$ are precompact. Let $g\in G$. For almost all $x\in X$ with $x\in A, gx\in B$ and $\omega(g,x)\in L$,
  we can make the following computation
  \[g=s(gx)\beta(\omega(g,x),u(x))s(x)^{-1}\in s(B)\left\{\beta(h,y)\left\vert y,hy\in B_{n,k}, h\in L\right.\right\}s(A)^{-1}.\]
  Observe that this last set is precompact and denote it by $K(L,A,B)\subset G$.
  Whenever the set \hbox{$\{x\in A\cap g^{-1}B\mid \omega(g,x)\in L\}$} is non-null, we see that $g\in K(L,A,B)$.
\end{proof}

\def\niets{
\section{Measure correspondences give rise to proper cocycles}
In this section, we show that measure correspondences give rise to proper cocycles.
The proof is based on recent work by Becker \cite{Becker:cocycles},
where he shows that every Borel cocycle is cohomologous to an essentially continuous cocycle.

\begin{definition}[{\cite[definition ...]{Becker:cocycles}}]
  \item A cocycle $\omega:G\times X\rightarrow H$ is essentially continuous if the following data exists.
    \begin{itemize}
    \item a Polish topology $\tau$ on $X$ that is compatible with the Borel structure of $X$ and such that
      the action $G\actson X$ is continuous with respect to $\tau$.
    \item a set $S\subset G\times X$
      such that $\omega\restrict S$ is continuous with respect to $\tau$ and for every $g\in G$, the set $S_g=\{x\in X\mid (g,x)\in S\}$
      has full measure in $X$.
    \end{itemize}
\end{definition}
In his paper, Becker calls this type of cocycle ``quintessentially continuous'', reserving the terminology
``essentially continuous'' for the case where the topology $\tau$ is given.

\def\niets{
\begin{definition}
  \label{def:MC}
  Let $G,H$ be two \lcsc groups and let $G\times H\actson (\Omega,\eta)$ be a non-singular action.
  We say that $\Omega$ is a measure correspondence if there exist standard measure spaces $(X,\mu)$ and
  $(Y,\nu)$ and (almost everywhere defined) measure class preserving Borel isomorphisms
  $\varphi:X\times H\rightarrow \Omega$ and $\psi:Y\times G\rightarrow \Omega$ such that for all $h\in H$
  we have that $\varphi(x,hk)=h\varphi(x,k)$ for almost all $(x,k)\in X\times H$,
  and such that for every $g\in G$, we have that $\psi(x,gh)=g\psi(x,h)$ for almost all $(x,h)\in Y\times G$.
\end{definition}

\begin{example}
  \label{ex:MC}
  Let $G\times H\actson (\Omega,\eta)$ be a measure correspondence of \lcsc groups, and consider measure preserving Borel isomorphisms $\varphi:X\times H\rightarrow \Omega$
  and $\psi:Y\times G\rightarrow\Omega$ as in definition \ref{def:MC}. Then there exist a unique non-singular action $G\actson (X,\mu)$ and cocycle $\omega:G\times X\rightarrow H$
  such that for every $g\in G$ we have that
  \[g\varphi(x,h)=\varphi(gx,h\omega(g,x)^{-1})\text{ for almost all }(x,h)\in X\times H.\]
  Similarly, we find a non-singular action $H\actson (Y,\nu)$ and a cocycle $\beta:H\times Y\rightarrow G$ such that for every $h\in H$, we get that
  \[h\psi(y,g)=\psi(hy,g\beta(h,y)^{-1})\text{ for almost all }(y,g)\in Y\times G.\]
  The these cocycles $\omega$ and $\beta$ are proper cocycles.
\end{example}

\begin{definition}
  \label{def:MEsub}
  Let $G,H$ be \lcsc groups. We say that $G$ is a measure equivalence subgroup of $H$ if there is a measure correspondence $(\Omega,\eta)$ between
  the two groups such that there is a $G$-invariant probability measure on $X$ in the class of $\mu$, where the action $G\actson (X,\mu)$ is as defined in example \ref{ex:MC}.
\end{definition}}

\begin{theorem}
  \label{thm:MC}
  Let $G,H$ be \lcsc groups and let $(\Omega,\eta)$ be a measure correspondence between $G$ and $H$.
  Consider the non-singular actions $G\actson (X,\mu)$
  and $H\actson (Y,\nu)$ associated with $\Omega$, together with the cocycles
  $\omega:G\times X\rightarrow H$ and $\beta:H\times Y\rightarrow G$.
  Then both $\omega$ and $\beta$ are proper cocycles.
\end{theorem}
\begin{proof}
  By symmetry, we only have to show that $\omega$ is a proper cocycle. By results of Mackey, we can assume
  that the actions $G\actson (X,\mu)$
  and $H\actson (Y,\nu)$ are everywhere defined Borel actions \cite{Mackey:actions}.
  Becker showed in \cite[corollary 1.4.8]{Becker:cocycles} that
  $\omega,\beta$ are cohomologous to cocycles that are essentially continuous.
  So by observation \ref{obs:cohom}, we can assume that $\omega$ and $\beta$ are essentially continuous.
  Moreover, by the Mackey cocycle theorem, we can assume that $\omega$ and $\beta$ are 
  strict Borel cocycles. This means that the cocycle relation
  $\omega(gh,x)=\omega(g,hx)\omega(h,x)$
  holds for all $g,h\in G$ and all $x\in X$ instead of almost all $x\in X$, and similarly for $\beta$.
  The cocycles $\omega,\beta$ remain essentially continuous, so we find the following data:
  \begin{itemize}
  \item Polish topologies on $X$ and $Y$ that are compatible with the Borel structures on $X$ and $Y$, and such that the actions $G\actson X$
    and $H\actson Y$ are continuous.
  \item Borel sets $S\subset G\times X$ and $T\subset H\times Y$ such that
    $\omega\restrict S$ and $\beta\restrict T$ are continuous and all the sets $S_g=\{x\in X\mid (g,x)\in S\}$ and $T_h=\{y\in Y\mid (h,y)\in T\}$
    with $g\in G$ and $h\in H$ have measure 1.
  \end{itemize}

  The proof of theorem \ref{thm:MC} proceeds in four steps.
  \begin{step}
    For every $\varepsilon>0$ and for every open neighborhood $e\in L\subset H$, there exists an open neighborhood $e\in U\subset G$
    and a Borel set $A\subset X$ with measure $\mu(A)>1-\varepsilon$ such that for every $g\in U$ we have that
    \[\omega(g,x)\subset L\text{ for almost all }x\in A.\]
  \end{step}
  Let $\varepsilon>0$ and fix any precompact open neighborhood $e\in L\subset H$. We know that $\omega(e,x)=e$ for all $x\in X$. By continuity of $\omega\restrict S$
  we see that for every $x\in X$ there are open neighborhoods $x\in W_x\subset X$ and $e\in V_x\subset G$ such that $\omega((V_x\times W_x)\cap S)\subset L$.
  The open sets $(W_x)_{x\in X}$
  cover $X$, and since $X$ is second countable, we find a countable subcover $(W_{x_n})_{n\in\IN}$. The $\sigma$-additivity of $\mu$ gives us
  a number $N$ such that the set $A=W_{x_1}\cup\ldots\cup W_{x_N}$ has measure $\mu(A)>1-\varepsilon$. Consider the open neighborhood $U=V_{x_1}\cap\ldots\cap V_{x_N}$ of identity in $G$.
  Then we see that for every $g\in U$ and almost every $x\in A$ we have that $\omega(g,x)\in L$.

  \begin{step}
    There exists a large family $\cA_0$ of Borel sets in $X$ such that for every pair $A,B\in\cA$ and every compact set $K\subset G$, the following holds.
    There is a precompact set $L\subset H$ such that for all $g\in K$ and all $x\in A$ with $gx\in B$ we have that $\omega(g,x)\in L$.
  \end{step}
  Observe that step 2 essentially states that condition (1) of definition \ref{def:propercocycle} holds for $\omega$, but
  $\omega(g,x)$ is in $L$ for all $x\in A\cap g^{-1}B$ instead of almost all. This is crucial for step 4.
  Fix a precompact open neighborhood $e\in L_0\subset H$.
  By step 1, we find an increasing sequence of Borel sets $(A_n)_n$ in $X$ and a decreasing sequence of open neighborhoods $e\in U_n\subset G$
  such that $\bigcup_n A_n$ has measure 1 and such that for all $n\in \IN$ and all $g\in U_n$ we have that $\omega(g,x)\in L_0$ for almost all $x\in A_n$.
  Consider the set $C_0$ of all $x\in\bigcup_n A_n$ such that for all $n\in\IN$ with $x\in A_n$, we have that $\omega(g,x)\in L$ for almost all $g\in U_n$, i.e.\@
  \[C_0=\{x\in \bigcup_n A_n \mid x\in A_n\Rightarrow \omega(g,x)\in L_0\text{ for almost all }g\in U_n\}.\]
  Observe that the set $C_0$ has measure $1$. For all $x\in X$, we consider the set $V_x\subset G$ that consists of all $g\in G$
  such that for all $n\in\IN$ with $x\in A_n$ and $g\in U_n$, we have that $\omega(g,x)\in L_0$. By definition of $C_0$, we see that
  $V_x$ has full measure in $G$ for all $x\in C_0$.

  Fix a countable dense subset $\{g_k\}_{k\in\IN}$ in $G$. We can assume that $g_1=e$. Consider the family $\cA$ of all Borel sets $A\subset X$
  that satisfy the following conditions for all $k\in\IN$
  \begin{itemize}
  \item $g_kA\subset C_0$
  \item $\omega(g_k,A)$ is precompact
  \item there exists $n\in\IN$ such that $\bigcup_{l=1}^k g_lA\subset A_n$.
  \end{itemize}
  We show that this family satisfies the conclusion of step 2. Let $A,B\in \cA$ and let $K\subset G$ be a compact set.
  It follows that $A\subset A_n$ for some $n\in\IN$. Moreover, we find $k\in\IN$ such that $K\subset \bigcup_{l=1}^kg_l^{-1}U_n$.
  There is a $m\in\IN$ such that $\bigcup_{l=1}^kg_lB\subset A_m$. Observe that the set $L=\bigcup_{l=1}^k \omega(g_l,A)^{-1}L_0^{-1}L_0$
  is precompact. Let $g\in K$ and $x\in A$ be such that $gx\in B$. Then we find some $l\leq k$ such that $g_lg\in U_n$. In particular,
  we see that $U_n(g_lg)^{-1}\cap U_m$ is a non-empty open set, so it has non-zero measure in $G$. In particular, we find $h_1\in V_x\cap U_n$
  and $h_2\in V_{g_lgx}\cap U_m$ such that $g=g_l^{-1}h_2^{-1}h_1$. We compute that
  \[\omega(g,x)=\omega(g_l^{-1}h_2^{-1}h_1,x)=\omega(g_l,gx)^{-1}\omega(h_2,g_lgx)^{-1}\omega(h_1,x)\in L.\]
  This finishes the proof of step 2.

  \begin{step}
    There exist Borel maps $s:X\rightarrow G$ and $u:X\rightarrow Y$ such that for all $g\in G$ we have that
    \begin{align*}
      u(gx)&=\omega(g,x)\text{ for almost all }x\in X
      g&=s(gx)^{-1}\beta(\omega(g,x),u(x))s(x)\text{ for almost all }x\in X.
    \end{align*}
    Moreover, for every set $V\subset Y$ with measure $1$, there is $h\in H$ such that $u^{-1}(h^{-1}V)$ has measure 1.
  \end{step}
  Fix measure class preserving
  Borel isomorphisms $\varphi:X\times H\rightarrow \Omega$ and $\psi:Y\times G\rightarrow \Omega$ as in definition \ref{def:MC}.
  Because $\psi$ commutes with the $H$-action, we find unique Borel maps $s:X\rightarrow G$ and $u:X\rightarrow Y$
  such that
  \[\psi(x,h)=h\varphi(u(x),s(x))\text{ for almost all }(x,h)\in X\times H.\]

  For every $g\in G$, we compute that for almost all $(x,h)\in X\times H$,
  \begin{align*}
    g\psi(x,h)&=\psi(gx,h\omega(g,x)^{-1})\\
    \Vert\qquad&\qquad\Vert\\
    h\varphi(u(x),gs(x))&=h\omega(g,x)^{-1}\varphi(u(gx),s(gx))\\
    \Vert\qquad&\qquad\Vert\\
    \varphi(hu(x),gs(x)\beta(h,u(x))^{-1})&=\varphi(h\omega(g,x)^{-1}u(gx),s(gx)\beta(h\omega(g,x)^{-1},u(gx))^{-1})
  \end{align*}
  It follows that for all $g\in G$ we have that for almost all $(x,h)\in X\times H$
  \begin{align*}
    u(gx)&=\omega(g,x)u(x)\text{ for almost all }x\in X\\
    gs(x)\beta(h,u(x))^{-1}&=s(gx)\beta(\omega(g,x)^{-1},u(gx))^{-1}\beta(h,\omega(g,x)^{-1}u(gx))^{-1}\\
    &=s(gx)\beta(\omega(g,x),u(x))\beta(h,u(x))^{-1}
  \end{align*}
  This finishes the proof of step 3.

  \begin{step}
    The cocycle $\omega$ is a proper cocycle.
  \end{step}
  By step 2, we find a large family $\cA_1$ of Borel sets in $X$ such that for every $A,B\in\cA_1$ and every compact $K\subset G$ we find a compact $L\subset H$
  such that for all $g\in K$ and $x\in A$ with $gx\in B$ we have that $\omega(g,x)\in L$. We find a similar large family $\cA_2$ of Borel sets in $Y$
  for the cocycle $\beta$. By step 3, we find Borel maps $s:X\rightarrow G$ and $u:X\rightarrow Y$ such that for all $g\in G$
  we have that
  \begin{align*}
    u(gx)&=\omega(g,x)\text{ for almost all }x\in X
    g&=s(gx)^{-1}\beta(\omega(g,x),u(x))s(x)\text{ for almost all }x\in X.
  \end{align*}
  Moreover, there is an $h_0\in H$ such that the family
  \[\cA_3=\{u^{-1}(h_0^{-1}A)\mid A\in\cA_2\}\]
  is a large family. Consider the large family of all Borel sets $A\subset X$ that satisfy the following conditions
  \begin{itemize}
  \item $A\in \cA_1$ and $A\in\cA_3$, so $h_0u(A)\in \cA_2$.
  \item $\beta(h_0,u(A))$ and $s(A)$ are precompact sets in $G$.
  \end{itemize}
  We show that $\omega$ is a proper cocycle with respect to this family $\cA$. Since $\cA$ is contained in $\cA_1$,
  it is clear that $\omega$ satisfies condition (1) from definition \ref{def:propercocycle}. To show that it also satisfies condition
  (2), let $A,B\in \cA$ and let $L\subset H$ be a compact set. By step 2, we find a compact set $K_0\subset G$ such that
  for all $h\in h_0Lh_0^{-1}$ and $y\in h_0u(A)$ with $hy\in h_0u(B)$, we have that $\beta(h,y)\in K_0$. Consider the precompact set
  $K=s(B)^{-1}\beta(h_0,u(B))^{-1}K_0\beta(h_0,A)s(A)$.
  Let now $g\in G$, we see that
  for almost all $x\in A$ with $gx\in B$ and $\omega(g,x)\in L$, we can make the following computation
  \begin{align*}
    g&=s(gx)^{-1}\beta(\omega(g,x),u(x))s(x)\\
    &=s(gx)^{-1}\beta(h_0,u(gx))^{-1}\beta(h_0\omega(g,x)h_0^{-1},h_0u(x))\beta(h_0,u(x))s(x)\\
    &\in K.
  \end{align*}
  So whenever the set $\{x\in A\cap g^{-1}B\mid \omega(g,x)\in L\}$ is non-null, we see that $g\in K$.
\end{proof}}

\section{Property A for pairs}
Jolissaint showed in \cite{Jolissaint:propercocycle} that, when $\omega:G\times X\rightarrow H$ is a proper cocycle, $(G,X)$ is an amenable pair and $H$ has the Haagerup property,
then $G$ has the Haagerup property. We will show that the same is true for property A and coarse embeddability instead of the Haagerup property. For property A
and coarse embeddability, we will be able to weaken the condition that $(G,X)$ is an amenable pair to the case where the pair $(G,X)$ has property A.
For the convenience of the reader, we also review the definition of an amenable pair, see \cite{Eymard:Means, Greenleaf:AmenableActions, Zimmer:AmenablePairs, Jolissaint:AmenablePairs}.
\begin{definition}
  \label{def:pairA}
  Let $G\actson (X,\mu)$ be a non-singular action of a \lcsc group on a standard probability space. Consider the Koopman representation
  $\pi:G\rightarrow U(\Lp^2(X,\mu))$.
  \begin{itemize}
  \item We say that the pair $(G,X)$ is amenable if there is a sequence of almost invariant unit vectors $\xi_n\in\Lp^2(X,\mu)$, i.e.\@
    \[\norm{\xi_n-\pi(g)\xi_n}\rightarrow 0\text{ uniformly on compact subsets of }G.\]
  \item Let $\cA$ be a large family of Borel subsets of $X$. We say that the pair $(G,X)$ has property A with respect to $\cA$ if
    for every compact set $K\subset G$ and every $\varepsilon>0$, there exists a continuous family $(\xi_g)_{g\in G}$ of unit vectors in $\Lp^2(X,\mu)$ such that
    \begin{itemize}
    \item $\norm{\xi_g-\xi_h}_2<\varepsilon$ whenever $g^{-1}h\in K$
    \item there is a set $A\in\cA$ such that each $\xi_g$ is supported in $gA$.
    \end{itemize}
  \end{itemize}
\end{definition}

The standard examples are the following. Suppose that $H\subset G$ is a closed normal subgroup and denote $Q=G/H$. Then the the pair $(G,Q)$ is an amenable pair if and only if
the group $Q$ is amenable. The group $Q$ has property A if and only if the pair $(G,Q)$ has property A with respect to the family of all precompact Borel sets in $Q$.

If $\mu$ is a $G$-invariant probability measure on $X$, then we see that the pair $(G,X)$ is amenable, because the constant function $1:X\rightarrow \IC$
is an invariant vector in $\Lp^2(X,\mu)$.
\begin{lemma}
  Let $G\actson (X,\mu)$ be a non-singular action of a \lcsc group on a standard probability space.
  If $(G,X)$ is an amenable pair, then $(G,X)$ has property A with respect to any large family $\cA$.
\end{lemma}
\begin{proof}
  Let $K\subset G$ be compact and let $\varepsilon>0$. Since $(G,X)$ is an amenable pair, we find a unit vector $\xi\in \Lp^2(X,\mu)$
  such that $\norm{\xi-\pi(g)\xi}<\frac\varepsilon3$ for all $g\in K$. Because $\cA$ is a large family, we find a set $A\in\cA$
  such that $\xi_0=\chara_A\xi$ satisfies $\norm{\xi_0-\xi}<\frac\varepsilon3$. Set $\xi_g=\pi(g)\xi_0$, then we see that $(\xi_g)_{g\in G}$ is a continuous
  family of unit vectors where every $\xi_g$ is supported in $gA$. Moreover, we compute that
  \[\norm{\xi_g-\xi_h}=\norm{\xi_0-\pi(g^{-1}h)\xi_0}\leq2\norm{\xi-\xi_0}+\norm{\xi-\pi(g^{-1}h)\xi}<\varepsilon\]
  for all $g,h\in G$ with $g^{-1}h\in K$.
\end{proof}
\def\niets{
 Roe's property A
for proper metric spaces with bounded geometry.

\begin{definition}[{\cite{R05}}]
  Let $(X,d)$ be a proper metric space, i.e.\@ all closed balls $\overline B(x,R)$ in $X$ are compact. In particular, $X$ is locally compact and $\sigma$-compact.
  Denote $\Prob(X)$ for the set of all Radon probability measures on $X$. We consider $\Prob(X)$ as a weak-$\ast$ compact convex subset of $C_0(X)^\ast$.
  \begin{itemize}
  \item The space $(X,d)$ has bounded geometry if there exists a radius $\varepsilon>0$ such that for all radii $R>0$ there is a natural number $N\in\IN$
    such for every $x\in X$, the ball $B(x,R)$ can be covered by $N$ balls of radius $\varepsilon$.
  \item A bounded geometry space $(X,d)$ is said to have property A if for every radius $R>0$ and for every $\varepsilon>0$ there is a weak-$\ast$ continuous
    map $\eta:X\rightarrow \Prob(X)$ such that $\norm{\eta_x-\eta_y}<\varepsilon$ whenever $d(x,y)< R$ and there is a radius $S>0$ such that each measure $\eta_x$
    is supported in $B(x,S)$.
  \end{itemize}
\end{definition}
For our purpose, the most important examples of proper metric spaces with bounded geometry come from \lcsc groups $G$. Remember that
there exists a proper left-invariant metric $d$ on $G$ such that $(G,d)$ has bounded geometry (see \cite{S74,HP06}). This metric is unique up to coarse equivalence.
It is easy to see that $G$ has property A as a \lcsc group if and only if $(G,d)$ has property A as a proper metric space (see \cite{DL:A}). Let $H\subset G$
be a closed subgroup. The metric $d$ on $G$ induces a $G$-invariant metric $d^\prime$ on the quotient space $G/H$. If the metric space $(G/H,d^\prime)$
has property A, the we will say that $H$ has property co-A in $G$. When $H$ is a normal subgroup, this is the case if and only if the quotient group $G/H$ has property A.

\begin{lemma}
  Let $(X,d)$ be a proper metric space with bounded geometry. Suppose that a \lcsc group $G$ acts on $X$ by isometries. If $(X,d)$ has Roe's
  property A, then the pair $(G,X)$ has property A with respect to the large family $\cA$ of all bounded Borel sets in $X$,
  for any quasi-invariant Radon probability measure $\mu$ on $X$.
\end{lemma}
\begin{proof}
  In order to show that the pair $(G,X)$ has property A, let $K\subset G$ be a compact set and let $\varepsilon>0$.
  Fix a point $x_0\in X$. Because the action $G\actson X$ is continuous, we find $R>0$ such that $d(gx_0,x_0)<R$ for every $g\in K$.
  Since $(X,d)$ has Roe's property A, there is a weak-$\ast$ continuous map $\eta:X\rightarrow\Prob(X)$ such that $\norm{\eta_x-\eta_y}<\varepsilon$
  whenever $d(x,y)<R$ and such that there is a radius $S>0$ with the property that $\supp\eta_x\subset B(x,S)$ for all $x\in X$.

  Take any continuous function $f:X\times X\rightarrow \IR_+$ that satisfies the following two conditions:
  \begin{itemize}
  \item there is a radius $r>0$ such that $f(x,y)=0$ whenever $d(x,y)>r$.
  \item for every $y\in X$, we get that the integral $\int_X f(x,y)\D\mu(x)=1$.
  \end{itemize}
  For any $x\in X$, we write $f_x:X\rightarrow \IR_+$ for the continuous compactly supported function that is given by $f_x(y)=f(x,y)$.

  Define Borel functions $\xi_g:X\rightarrow \IR_+$ the relation that
  \[\xi_g(x)^2=\int_X f(x,y) \D\eta_{gx_0}(y)=\eta_{gx_0}(f_x)\text{ for all }g\in G\text{ and }x\in X.\]
  Once we show the following properties of the family $(\xi_g)_{g\in G}$, we will have shown that the pair $(G,X)$ has property A.
  \begin{enumerate}
  \item each $\xi_g$ is a unit vector in $\Lp^2(X,\mu)$
  \item each $\xi_g$ is supported in $gB(x_0,S+r)=B(gx_0,S+r)$.
  \item the map $g\mapsto\xi_g$ is a continuous map $\xi:G\rightarrow \Lp^2(X,\mu)$.
  \item $\norm{\xi_g-\xi_h}<\varepsilon$ for all $g,h\in G$ with $g^{-1}h\in K$.
  \end{enumerate}

  To show the first property, observe that
  \[\norm{\xi_g}_2^2=\int_X\int_X f(x,y)\D\eta_{gx_0}(y)\D\mu(x)=\int_X\int_X f(x,y)\D\mu(x)\D\eta_{gx_0}(y)=\int_X 1\D\eta_{gx_0}(y)=1\]
  for all $g\in G$.

  For the second property, let $g\in G$ and suppose that $x\in X$ is such that $\xi_g(x)\not=0$. Then we know that
  \[\emptyset\not=\supp f_x\cap\supp \eta_{gx_0}\subset B(x,r)\cap B(gx_0,S).\]
  It follows that $x\in B(gx_0,S+r)$.

  The third property follows from the following argument. Suppose that $(g_n)_{n\in\IN}$ is a sequence in $G$ that converges to $g$.
  Then we know that $g_nx_0$ converges to $gx_0$, and in particular, the sequence $(g_nx_0)_{n\in\IN}$ is bounded. So we find a radius $s>0$
  such that $g_nx_0\in B(gx_0,s)$ for all $n\in\IN$. Observe that the restricted function $f\restrict{B(gx_0,S+r+s)\times X}$ is continuous and compactly
  supported. So we find an $M>0$ such that $\norm{f_x}_\infty\leq M$ for all $x\in B(gx_0,S+r+s)$. Since the map $\eta:X\rightarrow \Prob(X)$
  is weak-$\ast$ continuous, we see that $\xi_{g_n}(x)^2=\eta_{g_nx_0}(f_x)$ tends to $\eta_{gx_0}(f_x)=\xi_g(x)^2$ pointwise.
  Moreover, for all $n\in\IN$ and $x\in X$, we have that $\xi_{g_n}(x)^2\leq M$. By Lebesgue's dominated convergence theorem, we see that $\norm{\xi_{g_n}-\xi_{g}}_2$
  converges to $0$ when $n\rightarrow\infty$. In other words, $g\mapsto\xi_g$ is continuous.

  In order to show the last condition, we we observe that for all $g,h\in G$ with $g^{-1}h\in K$ we have that $d(gx_0,hx_0)<R$. Let $\eta$ be any Radon measure on $X$
  such that both $\eta_{gx_0}$ and $\eta_{hx_0}$ are absolutely continuous with respect to $\eta$. Denote the Radon-Nikodym derivatives by $\varphi_g,\varphi_h:X\rightarrow\IR_+$.
  Then we see that
  \begin{align*}
    \norm{\xi_g-\xi_h}_2^2&=\int_X\abs{\int_X f(x,y)(\varphi_g(y)-\varphi_h(y)\D\eta(y)}\D\mu(x)\\
    &\leq \int_X\int_X f(x,y)\D\mu(x)\abs{\varphi_g(y)-\varphi_h(y)}\D\eta(y)\\
    &= \int_X\abs{\varphi_g(y)-\varphi_h(y)}\D\eta(y)\\
    &=\norm{\eta_{gx_0}-\eta_{hx_0}}<\varepsilon
  \end{align*}
\end{proof}
}

\section{Proof of the main results}
\label{sect:proofs}
This section is devoted the the proofs of theorems \ref{thm:amen:intro} and \ref{thm:main:intro}. Theorem \ref{thm:amen:intro}
is essentially the same as the main result of \cite{Jolissaint:permanence}, but we use a slightly weaker notion of proper cocycle. Although the proof of theorem \ref{thm:amen:intro}
is very similar to the proof of \cite{Jolissaint:permanence}, we include a proof to assure the reader that \cite{Jolissaint:permanence} remains valid for our weaker notion of proper cocycle.
This shows that the Haagerup property, weak amenability and the weak Haagerup property pass to measure equivalence subgroups.

\begin{theorem}[theorem \ref{thm:amen:intro}, see also {\cite{Jolissaint:permanence}}]
  \label{thm:amen:proof}
  Let $G\actson (X,\mu)$ be a non-singular action of a \lcsc group. Let $\omega:G\times X\rightarrow H$ be a proper cocycle with values in another \lcsc group $H$.
  Assume that the pair $(G,X)$ is amenable. If $H$ has the Haagerup property (respectively is weakly amenable respectively has the weak Haagerup property),
  then so has $G$. Moreover, the weak amenability and weak Haagerup constants satisfy $\Lambda_{WA}(G)\leq\Lambda_{WA}(H)$ and $\Lambda_{WH}(G)\leq\Lambda_{WH}(H)$.
\end{theorem}
\begin{proof}
  Suppose that $H$ has the Haagerup property (respectively is weakly amenable with constant $\Lambda_{WA}(H)<\Lambda$, respectively has the
  weak Haagerup property with constant $\Lambda_{WH}(H)<\Lambda$). We have to show that $G$ has the Haagerup property (respectively
  is weakly amenable with constant $\Lambda_{WA}(H)<\Lambda$, respectively has the weak Haagerup property with constant $\Lambda_{WH}(H)<\Lambda$).

  Let $K\subset G$ be compact and let $\varepsilon>0$. Since $\omega$ is a proper cocycle, it is proper with respect to a large family $\cA$
  of Borel subsets of $X$. Denote the Koopman representation of $G\actson (X,\mu)$ by $\pi:G\rightarrow \cU(\Lp^2(X,\mu))$.
  Amenability of the pair $(G,X)$ gives a unit vector $\xi\in\Lp^2(X,\mu)$ such that $\norm{\xi-\pi_g\xi}_2<\varepsilon$ for all $g\in K$. We can assume that
  $\xi$ is supported in a set $A\in\cA$. Properness of the cocycle $\omega$ gives us a compact set $L\subset H$ such that for all $g\in K$ we have that
  $\omega(g,x)\in L$ for almost all $x\in A\cap g^{-1}A$. We find a continuous function $f_0:H\rightarrow \IC$ such that
  \begin{itemize}
  \item $\abs{f_0(h)-1}<\varepsilon$ for all $h\in L$.
  \item (in the case of the Haagerup property) $f_0$ is of positive type.
  \item (in the case of weak amenability and the weak Haagerup property) $f_0$ is a Herz-Schur multiplier with norm $\norm{f_0}_{HS}\leq\Lambda$.
  \item (in the case of the Haagerup property and the weak Haagerup property) $f_0$ is a $C_0$ function.
  \item (in the case of weak amenability) $f_0$ is compactly supported.
  \end{itemize}

  Define a Borel function $f:G\rightarrow \IC$ by the formula
  \[f(g)=\int_X \overline{\xi(x)}\,(\pi_g\xi)(x)\,f_0(\omega(g,g^{-1}x))\D\mu(x)\text{ for all }g\in G.\]
  We have to show that $f$ satisfies the following properties.
  \begin{itemize}
  \item $\abs{f(g)-1}<2\varepsilon$ for all $g\in K$.
  \item (in the case of the Haagerup property) $f$ is of positive type.
  \item (in the case of weak amenability and the weak Haagerup property) $f$ is a Herz-Schur multiplier with norm $\norm{f}_{HS}\leq\Lambda$.
  \item $f$ is continuous.
  \item (in the case of the Haagerup property and the weak Haagerup property) $f$ is a $C_0$ function.
  \item (in the case of weak amenability) $f$ is compactly supported.
  \end{itemize}

  The first property follows from the following computation:
  \begin{align*}
    \abs{f(g)-1}&=\abs{\int_X \overline{\xi(x)}\,(\pi_g\xi)(x)\,f_0(\omega(g,g^{-1}x))\D\mu(x)-1}\\
    &\leq \int_X\abs{\xi(x)}\,\abs{(\pi_g\xi)(x)}\,\abs{f_0(\omega(g,g^{-1}x))-1}\D\mu(x)+\abs{\inprod{\xi}{\pi_g\xi}-1}
  \end{align*}
  When $g\in K$, we have chosen $\xi$ such that $\abs{\inprod{\xi}{\pi_g\xi}-1}<\varepsilon$. In the same case, for almost all $x\in X$
  with $\xi(x)\not=0$ and $(\pi_g\xi)(x)\not=0$, we have that $x\in A$ and $g^{-1}x\in A$ and hence that $\omega(g,g^{-1}x)\in L$.
  It follows that $\abs{f_0(\omega(g,g^{-1}x))-1}<\varepsilon$, so we also see that
  \[\int_X\abs{\xi(x)}\,\abs{(\pi_g\xi)(x)}\,\abs{f_0(\omega(g,g^{-1}x))-1}\D\mu(x)<\inprod{\abs{\xi}}{\pi_g\abs{\xi}}\varepsilon\leq\varepsilon.\]
  This implies that $\abs{f(g)-1}<2\varepsilon$.

  Observe that $f$ is of positive type if and only if there is a separable Hilbert space $\cH$ and a bounded continuous function
  $\eta:G\rightarrow \cH$ such that $f(g^{-1}h)=\inprod{\eta_g}{\eta_h}$. So conditions 2 and 3 follow from the property below.
  Moreover, continuity of $f$ then follows from \cite[Appendix A]{Haagerup:CBAP} (see also \cite[lemma C.1]{Knudby:WH}).
  \begin{itemize}
  \item If $\cH_0$ is a Hilbert space and $\eta^0,\zeta^0:G\rightarrow \cH_0$ are bounded continuous functions such that $f_0(g^{-1}h)=\inprod{\eta^0_g}{\zeta^0_h}$
    for all $g,h\in H$. Consider $\cH=\Lp^2(X,\mu)\otimes \cH_0$ and define $\eta,\zeta:G\rightarrow \cH$ by the formula
    \[\eta_g(x)=(\pi_g\xi)(x)\eta^0_{\omega(g,g^{-1}x)}\text{ and }\zeta_g(x)=(\pi_g\xi)(x)\zeta^0_{\omega(g,g^{-1}x)}\]
    for all $g\in G$ and $x\in X$. Then it follows that $f(g^{-1}h)=\inprod{\eta_g}{\zeta_h}$ for all $g,h\in G$ and moreover
    we have that $\norm{\eta}_{\infty}=\norm{\eta^0}_\infty$ and $\norm{\zeta}_{\infty}=\norm{\zeta^0}_\infty$.
  \end{itemize}
  We now prove this property. Observe that, for all $g,h\in G$ and almost all $x\in X$, we have that
  \[\omega(g^{-1}h,h^{-1}gx)=\omega(g,g^{-1}gx)^{-1}\omega(h,h^{-1}gx),\]
  and hence we compute that
  \begin{align*}
    f(g^{-1}h)&=\int_X\overline{\xi(x)}(\pi_{g^{-1}h}\xi)(x)f_0(\omega(g^{-1}h,h^{-1}gx))\D\mu(x)\\
    &=\int_X\overline{\xi(x)}(\pi_{g^{-1}h}\xi)(x)f_0(\omega(g,g^{-1}gx)^{-1}\omega(h,h^{-1}gx))\D\mu(x)\\
    &=\int_X\overline{(\pi_g\xi)(x)}(\pi_h\xi)(x)f_0(\omega(g,g^{-1}x)^{-1}\omega(h,h^{-1}x))\D\mu(x)\\
    &=\int_X\overline{(\pi_g\xi)(x)}(\pi_h\xi)(x)\inprod{\eta^0_{\omega(g,g^{-1}x)}}{\zeta^0_{\omega(h,h^{-1}x)}}\D\mu(x)\\
    &=\inprod{\eta_g}{\zeta_h}.
  \end{align*}
  For every $g\in G$, the norm of $\eta_g$ can be computed as follows:
  \begin{align*}
    \norm{\eta_g}^2&=\int_X\abs{(\pi_g\xi)(x)}^2\norm{\eta^0_{\omega(g,g^{-1}x)}}^2\D\mu(x)\\
    &\leq \norm{\pi_g\xi}_2^2\norm{\eta^0}_\infty^2=\norm{\eta^0}_\infty.
  \end{align*}

  The last two properties follow from the following argument. Let $\delta>0$ in the case of the 5th property and let $\delta=0$
  in the case of the last property. We know that $L_0=\{h\in H\mid \abs{f_0(h)}>\delta\}$ is precompact, and we have to show that
  $K_0=\{g\in G\mid \abs{f(g)}>\delta\}$ is precompact.
  Since $\omega$ is a proper cocycle, we find a precompact set $K_1$ such that $g\in K_1$ whenever the set
  \[\{x\in X\mid x\in A, gx\in A\text{ and }\omega(g,x)\in L_0\}\]
  is non-null. Suppose that $\abs{f(g)}>\delta$, then there is a non-null set of $x\in X$ such that $\xi(x)\not=0$,
  $(\pi_g\xi)(x)\not=0$ and $f_0(\omega(g,g^{-1}x))>\delta$. All these $x\in X$ satisfy $x\in A$, $g^{-1}x\in A$ and $\omega(g,g^{-1}x)\in L_0$.
  It follows that $g\in K_1$.
\end{proof}

\begin{theorem}[theorem \ref{thm:main:intro}]
  \label{thm:main:proof}
  Let $G\actson (X,\mu)$ be a non-singular action of a \lcsc group on a standard probability space, and let $\cA$ be a large family such that
  the pair $(G,X)$ has property A with respect to $\cA$. Let $\omega:G\times X\rightarrow H$ be a proper cocycle with respect to the same family $\cA$.
  If $H$ is a \lcsc group with property A (respectively is coarsely embeddable), then $G$ has property A (respectively, is coarsely embeddable).
\end{theorem}
\begin{proof}
  Denote the Koopman representation by $\pi:G\rightarrow U(\Lp^2(X,\mu))$. Suppose that $H$ has property A (resp.\@ is coarsely embeddable). We show that
  $G$ has property A (resp.\@ is coarsely embeddable). Let $K\subset G$ be a compact subset and let $\varepsilon>0$.

  Since the pair $(G,X)$ has property A, we find a continuous family $(\xi_g)_{g\in G}$ of unit vectors in $\Lp^2(X,\mu)$ such that
  \begin{itemize}
  \item $\norm{\xi_g-\xi_h}<\frac\varepsilon2$ whenever $g^{-1}h\in K$.
  \item there is a set $A\in \cA$ such that every $\xi_g$ is supported in $gA$.
  \end{itemize}
  
  Properness of the cocycle $\omega$ gives us a compact set $L\subset H$ such that for all $g\in K$ and almost all $x\in A\cap g^{-1}A$ we get that
  $\omega(g,x)\in L$.

  Because the group $H$ has property A (resp.\@ is coarsely embeddable), we find a continuous positive definite kernel $k_0:H\times H\rightarrow \IC$ such that
  \begin{itemize}
  \item $\abs{k_0(g,h)-1}<\frac\varepsilon2$ whenever $g^{-1}h\in L$
  \item (in the property A case) the set $\{g^{-1}h\mid k_0(g,h)\not=0\}$ is precompact in $H$
  \item (in the coarse embeddability case) for all $\delta>0$ we have that the set\\
    $\{g^{-1}h\mid \abs{k_0(g,h)}>\delta\}$ is precompact.
  \end{itemize}

  Define a kernel $k:G\times G\rightarrow \IC$ by the formula
  \[k(g,h)=\int_X\overline{\xi_g(x)}\xi_h(x)k_0(\omega(g,g^{-1}x),\omega(h,h^{-1}x))\D\mu(x).\]
  We show that this kernel satisfies the conditions of definition \ref{def:A} (resp.\@ definition \ref{def:CE}):
  \begin{enumerate}
  \item $k$ is a positive definite kernel
  \item $\abs{k(g,h)-1}<\frac\varepsilon2$ whenever $g^{-1}h\in K$
  \item (in the property A case) the set $\{g^{-1}h\mid k(g,h)\not=0\}$ is precompact in $G$
  \item (in the coarse embeddability case) for all $\delta>0$ we have that the set\\
    $\{g^{-1}h\mid \abs{k(g,h)}>\delta\}$ is precompact.
  \end{enumerate}
  
  To prove (1), let $g_1,\ldots,g_n\in G$ and let $c_1,\ldots,c_n\in \IC$. Then we see that
  \begin{align*}
    \sum_{i,j=1}^n\overline c_i c_jk(g_i,g_j)=\int_X\sum_{i,j=1}^n\overline{\xi_{g_i}(x)c_i}\,\xi_{g_j}(x)c_j\,k_0(\omega(g_i,g_i^{-1}x),\omega(g_j,g_j^{-1}x))\D\mu(x).
  \end{align*}
  The integrand in the right hand side is positive for every $x\in X$ separately, because $k_0$ is a positive definite kernel.

  To prove (2), let $g,h\in G$ be such that $g^{-1}h\in K$. Then we see that, for almost all $x\in X$,
  \[\omega(g,g^{-1}x)^{-1}\omega(h,h^{-1}x)=\omega(g^{-1},hh^{-1}x)\omega(h,h^{-1}x)=\omega(g^{-1}h,h^{-1}x).\]
  Moreover, when $\xi_g(x)\not=0\not=\xi_h(x)$, then we get that both $h^{-1}x$ and $g^{-1}h h^{-1}x=g^{-1}x$ are elements
  of $A$. It follows that $\omega(g^{-1}h,h^{-1}x)\in L$ for almost all $x\in X$ with $\xi_g(x)\not=0\not=\xi_h(x)$.
  In particular, we see that for almost all such $x\in X$, $\abs{k_0(\omega(g,g^{-1}x),\omega(h,h^{-1}x))-1}<\frac\varepsilon2$.
  We compute that
  \begin{align*}
    \abs{k(g,h)-1}&\leq \int_X\abs{\xi_g(x)}\,\abs{\xi_h(x)}\,\abs{k_0(\omega(g,g^{-1}x),\omega(h,h^{-1}x))-1}\D\mu(x)\\
    &\qquad+\abs{\int_X\xi_g(x)\xi_h(x)\D\mu(x)-1}\\
    &\leq \frac\varepsilon2\inprod{\abs{\xi_g}}{\abs{\xi_h}}+\abs{\inprod{\xi_g}{\xi_h}-1}\\
    &\leq \frac\varepsilon2\norm{\xi_g}\norm{\xi_h} + \norm{\xi_g-\xi_h}\norm{\xi_h}\\
    &<\varepsilon.
  \end{align*}

  In order to prove (3) and (4) we prove the following assertion
  \begin{itemize}
  \item Let $\delta\geq0$. If $\{g^{-1}h\in H\mid \abs{k_0(g,h)}>\delta\}$ is precompact, then also\\
    $\{g^{-1}h\in G\mid \abs{k(g,h)}>\delta\}$
    is precompact.
  \end{itemize}
  Statement (3) follows by setting $\delta=0$. Assume that $\widetilde L=\{g^{-1}h\in H\mid \abs{k_0(g,h)}>\delta\}$ is precompact.
  Since $\omega$ is a proper cocycle with respect to $\cA$, we find that
  \[\widetilde K=\left\{g\in G\bigg\vert\mu\{x\in A\cap g^{-1}A\mid \omega(g,x)\in\widetilde L\}>0\right\}\]
  is precompact. Suppose that $g,h\in G$ are such that $\abs{k(g,h)}>\delta$. It follows that there is a non-null Borel set $B\subset X$
  such that $\abs{k_0(\omega(g,g^{-1}x),\omega(h,h^{-1}x))}>\delta$ for all $x\in B$. It follows that $\omega(g^{-1}h,h^{-1}x)\in\widetilde L$ for all $x\in B$,
  and hence that $g^{-1}h\in \widetilde K$.
\end{proof}


\bibliography{references,bibliography}{}
\bibliographystyle{sdpalpha}
\end{document}

%% file: sdpDefinitions.tex
\usepackage{amsmath}
\usepackage{amssymb}
\usepackage{amsthm}
\usepackage{amscd}
\usepackage{tikz}
\usepgflibrary{arrows}
\newtheorem{theorem}{Theorem}[section]
\newtheorem{lemma}[theorem]{Lemma}

\newtheorem{proposition}[theorem]{Proposition}

\newtheorem{definition}[theorem]{Definition}

\newtheorem{example}[theorem]{Example}
\newtheorem{examples}[theorem]{Examples}

\newtheorem{observation}[theorem]{Observation}
\newcounter{step}[theorem]
\newenvironment{step}{\par\refstepcounter{step}\textbf{step \thestep: }\begingroup\it}{\endgroup\par}

\DeclareMathOperator{\Prob}{Prob}
\def\interior#1{\overset{\circ}{#1}}

\newcommand\norm[1]{\left\lVert #1\right\rVert}

\newcommand\abs[1]{\left\lvert #1\right\rvert}
\newcommand\inprod[2]{\left\langle #1, #2\right\rangle}
\DeclareMathOperator{\IC}{\mathbb{C}}
\DeclareMathOperator{\IR}{\mathbb{R}}
\DeclareMathOperator{\IRpos}{\mathbb{R}_{+}^{\ast}}
\DeclareMathOperator{\IQ}{\mathbb{Q}}
\DeclareMathOperator{\IZ}{\mathbb{Z}}
\DeclareMathOperator{\IN}{\mathbb{N}}

\DeclareMathOperator{\F}{F}

\DeclareMathOperator{\cH}{\mathcal{H}}
\DeclareMathOperator{\cU}{\mathcal{U}}

\DeclareMathOperator{\Circle}{S}

\def\MatrixGroup#1_#2#3{\mathchoice{#1_{#2}\hspace*{-0.3mm}#3}{#1_{#2}\hspace*{-0.4mm}#3}{#1_{#2}\hspace*{-0.4mm}#3\hspace*{0.3mm}}{#1_{#2}\hspace*{-0.4mm}#3\hspace*{0.3mm}}%
  \def\tempa{#3}%
  \def\tempb{\IQ}\ifx\tempa\tempb\else%
  \def\tempb{\IZ}\ifx\tempa\tempb\else%
  \def\tempb{R}\ifx\tempa\tempb\else%
  \def\tempb{(\F_2[X])}\ifx\tempa\tempb\else%
  \def\tempb{(\F_2)}\ifx\tempa\tempb\else%
  \def\tempb{\IC}\ifx\tempa\tempb\else%
  \fout\fi\fi\fi\fi\fi\fi%
}

\makeatletter
\def\quot{\@ifnextchar[{\sdp@quotleftA}{\sdp@quotright}}
\def\sdp@quotleftA[#1]#2{\@ifnextchar[{\sdp@quotboth[{#1}]{#2}}{\sdp@quotleft[{#1}]{#2}}}
\def\sdp@quotright#1[#2]{#1/#2}
\def\sdp@quotboth[#1]#2[#3]{#1\backslash #2 / #3}
\def\sdp@quotleft[#1]#2{#1\backslash #2}
\makeatother

\DeclareMathOperator{\Aut}{Aut}

\def\Autmp(#1,#2){\Aut_{#2}(#1)}
\def\Auttp(#1,#2){\Aut_{#2}(#1)}

\DeclareMathOperator{\Out}{Out}
\def\Outmp(#1,#2){\Out_{#2}(#1)}
\def\Outtp(#1,#2){\Out_{#2}(#1)}

\def\Rel(#1\actson #2){\mathcal{R}(#1\actson #2)}

\def\fullg(#1){\left[#1\right]}
\def\fullpg(#1){\left[\left[#1\right]\right]}

\DeclareMathOperator{\Lp}{L}

\DeclareMathOperator{\cA}{\mathcal{A}}

\def\vnInterB(#1){\zH_{#1}}

\DeclareMathOperator{\actson}{\curvearrowright}

\DeclareMathOperator{\supp}{supp}

\def\D{d}

\newcommand\restrict[1]{\vert_{#1}}
\makeatletter
\def\zH{\@ifnextchar_{\zH@}{\@undefined}}
\DeclareMathOperator{\zH@}{H}
\def\smallscripts#1{\@ifnextchar_{\smallscripts@{#1}}{\@undefined}}
\def\smallscripts@#1{\@ifnextchar_{\smallscripts@subP{#1}}{\@ifnextchar^{\smallscripts@superP{#1}}{\smallscripts@none{#1}}}}
\def\smallscripts@subP#1_#2{\@ifnextchar^{\smallscripts@subsuperP{#1}{#2}}{\smallscripts@sub{#1}{#2}}}
\def\smallscripts@subsuperP#1#2^#3{\smallscripts@both{#1}{#2}{#3}}
\def\smallscripts@superP#1^#2{\@ifnextchar_{\smallscripts@supersubP{#1}{#2}}{\smallscripts@super{#1}{#2}}}
\def\smallscripts@supersubP#1#2_#3{\smallscripts@both{#1}{#3}{#2}}
\def\smallscripts@none#1{{\displaystyle #1} }
\def\smallscripts@sub#1#2{{\displaystyle #1}_{\scriptscriptstyle #2}}
\def\smallscripts@super#1#2{{\displaystyle #1}^{\scriptscriptstyle #2}}
\def\smallscripts@both#1#2#3{{\displaystyle #1}_{\scriptscriptstyle #2}^{\scriptscriptstyle #3}}
\makeatother
\newcommand\chara{\smallscripts{\chi}}
